\documentclass[12pt]{article}
\usepackage{graphicx}
\usepackage[margin=1in]{geometry}
\usepackage{booktabs}
\usepackage{multirow}
\usepackage{siunitx}
\usepackage{authblk}
\usepackage{amsmath,amssymb,amsthm}
\usepackage[colorlinks=true, linkcolor=blue, citecolor=blue]{hyperref}

\newtheorem{theorem}{Theorem}
\numberwithin{theorem}{section}

\newtheorem{corollary}[theorem]{Corollary}

\newtheorem{remark}[theorem]{Remark}

\newcommand{\Cc}{{\mathbb C}}

\newcommand{\Rr}{{\mathbb R}}
\newcommand{\Zz}{{\mathbb Z}}
\newcommand{\Pp}{{\mathbb P}}
\newcommand{\sD}{{\mathcal D}}
\newcommand{\sW}{{\mathcal W}}

\newcommand{\suchthat}{\,|\,}

\newcommand{\fapprox}{{\tilde{f}}}
\newcommand{\sF}{{\mathcal F}}
\newcommand{\sG}{{\mathcal G}}
\newcommand{\sH}{{\mathcal H}}
\newcommand{\sN}{{\mathcal N}}
\newcommand{\Hessian}{{{\hbox{\rm Hessian}}}}
\newcommand{\NULL}{{{\hbox{\rm null~}}}}

\title{Robust Numerical Algebraic Geometry}
\author{Emma R. Cobian}
\author{Jonathan D. Hauenstein}
\author{Charles W. Wampler}
\affil{University of Notre Dame}
\date{}

\begin{document}

\maketitle

\begin{abstract}
\noindent The field of numerical algebraic geometry consists of 
algorithms for numerically solving systems of polynomial equations.
When the system is exact, such as having 
rational coefficients, 
the solution set is well-defined.  However,
for a member of a parameterized family of polynomial systems
where the parameter values may be measured with imprecision 
or arise from prior numerical computations, uncertainty may
arise in the structure of the solution set, including the number
of isolated solutions, the existence of higher dimensional solution
components, and the number of irreducible components along with their 
multiplicities. The loci where these structures change form a
stratification of exceptional algebraic sets in the space of parameters. We describe methodologies for
making the interpretation of numerical results more robust
by searching for nearby parameter values on an exceptional set. We demonstrate 
these techniques on several illustrative examples and 
then treat several more substantial problems
arising from the kinematics of mechanisms and robots.
\end{abstract}

\section{Introduction}

Numerical algebraic geometry concerns algorithms for numerically solving 
systems of polynomial equations, primarily based on homotopy methods,
often also referred to as polynomial continuation. 
Reference texts for numerical algebraic
geometry are \cite{BertiniBook,SW05}, and
software packages that implement its algorithms are available in \cite{Bertini,HomContJL2018,NAG4M2,PHCpack2011}.
Built on a foundation 
of methods for finding all isolated solutions, the field has grown
to include algorithms for computing the irreducible decomposition of 
algebraic sets along with operations 
such as membership testing, intersection, and projection.
The basic construct of the field is a \emph{witness set}, say~$\sW$, in
which a pure $D$-dimensional algebraic set, say $X\subset\Cc^n$, is represented by a structure having three members:
\begin{equation}\label{eq:witnessSet}
  \sW=\{f,L,W\}
\end{equation}
where $f:\Cc^n\rightarrow\Cc^k$ is a polynomial system such that $X$ is a
$D$-dimensional component of $V(f)=\{x\in\Cc^n\suchthat f(x)=0\}$, 
$L:\Cc^n\rightarrow\Cc^D$ is a 
\emph{slicing system} of $D$ generic linear polynomials,
and $W=X\cap V(L)$ is a \emph{witness point set} for $X$.
Given a witness set, one can sample the set it represents by moving its slicing system
in a homotopy. Given witness sets for two components, one can compute their intersection, 
obtaining witness sets for the components of the intersection \cite{HW}. 
A witness set can be decomposed into its irreducible components using monodromy \cite{Monodromy} 
to group together points on the same irreducible component and using a trace test~\cite{BHL,SparseTrace,HR15,TraceLRS,SVWtraces} to verify when this process is complete. Algorithms built on these and related techniques 
compute a numerical irreducible decomposition of $V(f)$,
producing a collection of witness sets, one for each irreducible component.
In a similar fashion, one may construct pseudo-witness sets for projections
of algebraic sets \cite{WitnessProjection,MembershipProjection}, 
which is the geometric equivalent of symbolic elimination. 
Altogether, using algorithms for intersection, union, projection, and membership testing, one can represent and manipulate constructible algebraic sets. 

Applications of algebraic geometry often involve parameterized families of
polynomial systems of the form $f(x;p):\Cc^n\times\Cc^m\rightarrow\Cc^k$, where $x$ is an array
of variables and $p$ is an array of parameters. For example, in kinematic analysis, $x$ may
be variables describing the relative displacement at joints between parts while $p$ may 
describe the length of links or the axis of a rotational joint. In kinematic synthesis,
where one seeks to find a mechanism to produce a desired motion, these roles may be interchanged. 
The long history of research in kinematics and its applications to mechanisms and robotics is
extensive; \cite{raghavan1995solving,AlgKin} provide useful overviews.
In multi-view computer vision, the variables of a scene reconstruction describe the location of
objects and cameras in three-dimensional space while the parameters are the coordinates of features
observed in the camera images \cite{agarwal2021multiview,fabbri2022trifocal,heyden1997algebraic,ma2004invitation}. 
In chemical equilibrium, concentrations are variables and
reaction rates are parameters \cite[Chap.~9]{morgan2009solving}
and \cite{dickenstein2020,perez2012chemical}. 
In short, in a single instance of the family,
variables are the unknowns while parameters are given,
and the parameter space defines a family of problems having the same polynomial structure.
In the simplest case, the parameters are merely the coefficients of polynomials with
a fixed set of monomials while, in many applications, the coefficients are often polynomial functions of
the parameters. 
In general, one can consider a parameter space that is an irreducible 
algebraic set or, with minor additional conditions, even a complex analytic set \cite{CoeffParam}.
However, for simplicity, we will assume here that 
the parameter space is the complex Euclidean space $\Cc^m$.

The algorithms of numerical algebraic geometry
compute using floating point arithmetic, so function evaluations and solution points
are consequently inexact. Furthermore, in applications, parameters may be values measured
with imprecision or they may be numerical values produced in prior stages of computation.
In light of these uncertainties, the results reported by the
algorithms may require interpretation. For example, a coordinate 
of a solution point computed as near zero may indicate that 
there is an exact solution nearby with that coordinate
exactly zero, but this is not assured. Moreover, if the polynomial system
is given with uncertain parameters or with coefficients presented as floating point values, there is some
uncertainty about which problem is being posed. The interpretation of
the solution set, such as classifying how many endpoints of a homotopy
have converged to finite solution points versus how many have diverged to infinity,
is subject to judgements about round-off errors from inexactly evaluating
inexactly specified polynomials. 
Similar judgements must be made when
categorizing singular versus nonsingular solutions, 
sorting solutions versus nonsolutions when utilizing randomization,
concluding that a witness set is complete via a trace test,
or counting multiplicities according to how many homotopy paths
converge to~(nearly) the same point. 
Our aim in this article is to lay out methodologies for making these
interpretations more robust yielding {\em robust numerical algebraic geometry}.

In a parameterized family, the irreducible decomposition has the same structure for generic parameters, 
that is, the sets of parameters where the structure changes lie on proper
algebraic subsets of the parameter space. 
By structure, we mean features over the complex numbers 
described by integers,
such as the number of irreducible components, their dimensions and degrees, and the multiplicity
of the witness points. 
Within any irreducible algebraic subset of parameter space, there may exist 
algebraic subsets of lower dimension where these integer features change again, creating a stratification of algebraic sets of successively greater specialization. In an early discussion of this phenomenon,
Kahan \cite{Kahan1972} called such sets \emph{pejorative manifolds},
but we prefer the more neutral term \emph{exceptional sets}.
Typically, an analyst would like to know that the problem they have 
posed is close to an exceptional set.  Moreover, since such sets have
zero measure in their containing parameter space, the fact that a posed problem 
falls quite close to an exceptional set may indicate that this is not a coincidence,
and in fact, the exceptional case 
is the true item~of~interest. Moreover, constraining the problem to lie on
an exceptional set  can convert an ill-conditioned problem into a
well-conditioned~one~\cite{Kahan1972}.

Factoring a single multivariate polynomial presents a special case of the phenomena we
presently address. Wu and Zeng \cite{wu2017numerical} note that factorization 
of such a polynomial is ill-posed when coefficient perturbations are considered
since the factors change discontinuously as the coefficients approach a 
factorization submanifold.
Their answer to this problem is to define a metric for the distance between 
polynomials and a partial ordering on the factorization structures. 
This partial ordering corresponds to algebraic set inclusion in the stratification of
factorization structures, and in the parlance of Wu and Zeng, each such set is said to be
more singular than any set that contains it. (Roughly, 
a polynomial with more factors is more singular than one with fewer, and for the same
degree, a factor appearing with multiplicity
is more singular that a product of distinct factors.) 
Wu and Zeng regularize the numerical factorization 
problem by requiring the user to provide an uncertainty ball around the given 
polynomial. Among all the possible factorization manifolds that intersect the uncertainty ball,
they define the numerical factorization to be the nearest polynomial on the exceptional set of highest
singularity. As the radius of the uncertainty ball grows, the numerical factorization may
change to one of higher singularity. So, while the problem remains ill-posed at 
critical radii, it has become well-posed everywhere else. 
In the case of a single polynomial of moderate degree, the number of possible specializations is small enough
that one can enumerate them all. This gives the potential of finding the unique numerical factorization
for most uncertainty radii, and a finite list of alternatives for any range of radii, 
functionalities provided by Wu and Zeng's software package.

For parameterized systems of polynomials, enumeration of the possible irreducible decomposition structures
is a formidable task, and we do not attempt it here. However, in the process of computing
a decomposition, there typically will be only a few places where the uncertainty in judging how to
classify points warrants exploration of the alternatives. The same can be said for more limited objectives,
such as computing only the isolated solutions of a system by homotopy, where one may question the number of solution 
paths deemed to have landed at infinity or observe a cluster of path endpoints that might indicate a single solution point of
higher multiplicity. 
If the uncertainty is due solely to 
floating point round-off and we have access to either 
a symbolic description 
or a refinable numerical description
of the parameter values, then the correct
judgement can be made with high confidence by increasing the precision of the computation. 
Such results might then be certified either by symbolic computations or by verifiable numerics,
such as Smale's alpha theory \cite{alphaTheory} (see also~\cite[Ch.~8]{BCSS})
or by techniques from interval analysis \cite{Moore79}.
Our concern in this article is for cases where there is inherent uncertainty in the parameters.
This may arise from empirical measurements of the parameters
or, perhaps, the parameters arise from prior
computations in finite precision. 
We assume that the questionable structural element has been identified
and our task is to find the nearest point 
in parameter space where the special structure~occurs.

After describing some robustness scenarios 
in Section~\ref{sec:Background}, 
Section~\ref{sec:Robust} provides a framework for robustness
in numerical algebraic geometry.  This framework is then 
applied to a variety of structures: fewer finite solutions
in Section~\ref{sec:FiniteRoots},
existence of higher-dimensional components in Section~\ref{sec:HigherDim},
components that further decompose into irreducible components
in Section~\ref{sec:Decompose}, and solution sets
of higher multiplicity in Section~\ref{sec:HigherMult}.
We formulate the algebraic conditions implied by each type of structure 
and use
local optimization techniques to find a nearby 
set of parameters satisfying them.
After treating each type of specialized structure and 
illustrating on a small example,
we show the effectiveness of the approach on three more substantial problems coming
from the kinematics of mechanisms and robots in Section~\ref{sec:Examples}.
A~short conclusion is provided in Section~\ref{sec:Conclusion}.
Files for the examples, which are all computed using 
{\tt Bertini} \cite{Bertini},
are available at \url{https://doi.org/10.7274/25328878}.

\section{Robustness scenarios}\label{sec:Background}

Before presenting our approach to robustifying numerical algebraic geometry,
we discuss scenarios where proximity to an exceptional set leads to ill-conditioning.

To abbreviate the discussion, we use the phrase ``with probability one'' as shorthand for
the more precise, and stronger, condition that exceptions are a proper algebraic
subset of the relevant parameter space, where the parameter space in question 
should be clear from context. Similarly, a point in 
parameter space is ``generic'' if it lies in the dense Zariski-open set that is
the complement of the set of exceptions. For example, the reference to ``a system of $D$
generic linear polynomials,'' $L:\Cc^n\rightarrow\Cc^D$, just after \eqref{eq:witnessSet} 
means a system of the form $Ax+b$ where matrix $[A\>\>b\,]$ has
been chosen from the dense Zariski-open subset of $\Cc^{D\times(n+1)}$ where
$V(L)$ intersects $X$ transversely. In numerical work, we make the assumption
that a random number generator suffices for choosing generic points.

Our discussion utilizes the concept of a \emph{fiber product} \cite{FiberProducts}. For algebraic sets
$A$ and $B$ with algebraic maps $\pi_A:A\rightarrow Y$ and $\pi_B:B\rightarrow Y$, the
fiber product of $A$ with $B$ over $Y$ is
\begin{equation}
    A \times_Y B = \{(a,b)\in A\times B \suchthat \pi_A(a)=\pi_B(b)\}.
\end{equation}
One may similarly form fiber products between three or more algebraic sets.
In this article, the maps involved in forming fiber products will all be
natural projections of the form $(x,p)\mapsto p$. Moreover, 
for polynomial systems $f,g:\Cc^n\times\Cc^m\rightarrow\Cc^k$, if $A$ and $B$ 
in $\Cc^n\times\Cc^m$ are
irreducible components of $V(f(x,p))$ and $V(g(x,p))$, respectively,
then the fiber product of $A$ with $B$ over $\Cc^m$ is an algebraic set in
$V(f(x_1,p_1),g(x_2,p_2),p_1-p_2)$, a so-called reduction to the diagonal
which is isomorphic to an algebraic set in
\mbox{$V(f(x_1,p),g(x_2,p))\subset \Cc^n\times\Cc^n\times\Cc^m$}. In this situation,
we refer to $\{f(x_1,p),g(x_2,p)\}$ as a fiber product system.
We note the convention used throughout is 
that $f(x,p)$ means that both $x$ and $p$
are considered as variables which is in contrast to $f(x;p)$
where $x$ are variables and $p$ are parameters.

\subsection{Multiplicity example} 
As a simple first example, consider solving $V(f)$ for the single polynomial  
$f=x^2+2\sqrt{2}x+2$, which has the factorization $(x+\sqrt2)^2$ and hence 
$V(f)$ is the single point $x=-\sqrt2$ of multiplicity 2. If instead of the
exact symbolic form of $f$, we are given an eight-digit version of it, say
\[
  \fapprox_8 = x^2 + 2.8284271x +2,
\]
the Matlab \verb|roots| command, operating in double precision, returns the two
roots
\[
x_8= -1.414213550000000 \pm 0.000187073241389\text{i}
\]
where $\text{i} = \sqrt{-1}$.
Using these roots as initial guesses for Newton's method, the solutions
of the sixteen-digit version of $f$,
\[
 \fapprox_{16} = x^2 + 2.828427124746190x +2,
\]
are computed in double precision as
\[
x_{16} =  -1.414213553589213 \pm 0.000000183520060\text{i}.
\]
One may work the problem in increasingly higher precision 
by considering a sequence of approximations $a_\ell$
of $2\sqrt{2}$ rounded off to $\ell$ digits.  For every $\ell$,
$\fapprox_\ell=x^2+a_\ell x+2$ has a pair of roots in the vicinity of $-\sqrt{2}$. 
A numerical package with adjustable precision
can refine~$\fapprox_\ell$ and a solution $x_\ell$ until 
$|x_\ell+\sqrt{2}|$ is smaller than any positive error 
tolerance one might set. 
Whether a computer program using this refinement process reports two isolated roots or one
double root depends on settings for its precision and tolerance.
If the same program is only given~$\fapprox_8$ or~$\fapprox_{16}$, the roots stay
distinct no matter what precision is used for~Newton's~method.

We can stabilize this situation by asking if there exists a nearby
polynomial with a double root. That is, we ask if there is a polynomial near to $\fapprox_8$
of the form $\hat f(x;p)=x^2+px+2$ with a root that also satisfies the derivative
$\hat f'(x;p)=2x+p$. Solving $V(\hat f(x,p),\hat f'(x,p))$ 
in~$\Cc^2$ with Newton's method 
and an initial guess taking $x=x_8$ and $p=a_8$, returns
\[
 (\hat x,\hat p)=(-1.414213562373095,2.828427124746190),
\]
where the imaginary parts have converged to zero within machine epsilon. Moreover,
the Jacobian matrix for this structured system is clearly nonsingular:
\[
  J\left[
\begin{array}{c}
     \hat f  \\
     \hat f' 
\end{array}
\right]
=
\left[
\begin{array}{cc}
     2x+p & x  \\
     2 & 1 
\end{array}
\right]
\approx
\left[
\begin{array}{cc}
     0 & -1.414213562373095  \\
     2 & 1 
\end{array}
\right].
\]
The good conditioning of this double-root problem not only 
leads to a quickly convergent, accurate answer, but also it could be
used to certify the answer via alpha theory \cite{alphaTheory} or interval analysis \cite{Moore79}.
The acceptance of the double root as the ``correct'' answer 
depends on whether $\hat p$ is within the tolerance
of the given data. After all, if the user really wants the roots
of $\fapprox_8$ as given, then $x_8$ is the better answer.  However, 
if the
coefficient on $x$ is acknowledged to be only known to eight digits,
then the double root $\hat x$ is the preferred answer. 

Suppose that we reformulate such 
that the parameter space has two entries, say
$$\tilde f(x;p)=x^2+p_1x+p_2,$$
then we may search for the system with a double root nearest to $\fapprox_8$ as
\[
  \min\|p-(2.8284271,2)\| \text{ subject to } (\tilde f(x,p),\tilde f'(x,p))=0.
\]
Using the Euclidean norm, the global optimum is attained at
the nonsingular point
\[
  (x,p_1,p_2) = (-1.414213558248730, \:  2.828427116497461, \:  1.999999988334534),
\]
which again shows that there is a polynomial $\tilde f$ near the given polynomial, $\fapprox_8$, such that $\tilde f$ has a double root.
This illustrates the flavor of the approach put forward in \cite{wu2017numerical}, although
they also treat multivariate polynomials, which have a richer set of
factorization structures than just multiplicity.

\subsection{Divergent solutions} 
In numerical algebraic geometry, one of the most common objectives is to find the isolated
solutions of a ``square'' polynomial system, say $f(x;p):\Cc^n\times\Cc^m\rightarrow\Cc^n$, at a given parameter point, say $p=p^*$. (Here, \emph{square} means the number of equations equals the number of variables.) A standard result of the field states that the number of isolated solutions is
constant for all $p$ in a nonempty open Zariski set in $\Cc^m$. In other words, the exceptions
are a proper algebraic subset of $\Cc^m$, say $P^*$. A key technique in the field uses this fact to build homotopies for finding all isolated points. In particular, if we have all isolated solutions, 
say~\mbox{$S_1\subset\Cc^n$}, of start system $f(x;p_1)$ for a generic set
of parameters, $p_1$, then the homotopy~$f(x;\phi(t))=0$ for a general enough continuous path $\phi(t):[0,1]\rightarrow\Cc^m$ with $\phi(1)=p_1$ and $\phi(0)=p_0$
defines $\#S_1$ continuous paths whose limits as
$t\rightarrow0$ include all isolated solutions of a target system $f(x;p_0)$,
e.g., see \cite{CoeffParam} or \cite[Thm.~7.1.1]{SW05}.
The conditions for a ``general enough'' path are very mild; in fact, $\phi(t)=tp_1+(1-t)p_0$ suffices with probability one when $p_1$ is chosen randomly, independent of $p_0$. Many \emph{ab initio} homotopies, which solve a system from scratch, fit into this mold. For example, polyhedral homotopies are formed by considering the family of systems having the
same monomials as the target system, so that the parameter space consists of the coefficients of the monomials
\cite{huber1995polyhedral,li1997numerical,verschelde1994homotopies}. After solving one $f(x;p_1)$ for generic $p_1$ by such a technique, one may proceed to solve any target system in the family by parameter homotopy.
If the target parameters are special, i.e., if $p_0\in P^*$, then $f(x;p_0)$ has fewer isolated solutions than $f(x;p_1)$
meaning that some solution paths of the homotopy either diverge to infinity
or some of the endpoints lie on a positive-dimensional solution component.  Diverging
to infinity can be handled by homogenizing $f$ and working on a projective space \cite{morgan1986transformation}.
Thus, paths that originally diverged to infinity are transformed to converge 
to a point with homogeneous coordinate equal to zero. 

When one executes a parameter homotopy using floating point arithmetic, the computed homogeneous coordinate of a divergent path is typically not exactly zero but rather some complex number near zero. This also occurs
when $p_0$ is slightly perturbed off of the special set $P^*$. Usually, one does not know the conditions that
define the algebraic set $P^*$, but even so, with a solution near infinity in hand, one may wonder how far $p_0$ is from $P^*$. Suppose that due to round-off or measurement error, $p_0$ is uncertain. Then it could be of high interest to know whether the closest point of $P^*$ is within the uncertainty ball centered on $p_0$.

It often happens that more than one endpoint of a homotopy falls near infinity. In such cases, it is of interest to find a nearby point in parameter
space where all those points land on infinity simultaneously. Let us assume that $f(x;p)$ has been homogenized so that $x$ has homogeneous coordinates 
$[x_0,x_1,\ldots,x_n]\in\Pp^n$, with $x_0=0$ being the hyperplane at infinity. To consider $j>1$ points simultaneously, we must introduce a 
double subscript notation, where point $x_i\in\Pp^n$ has coordinates $x_i=[x_{i0},\ldots,x_{in}]$. Point $x_i$ is a solution at infinity if it satisfies
the augmented system $F(x_i;p)=\{f(x_i;p),x_{i0}\}=0$, so sending $j>1$ points to infinity simultaneously requires
\begin{equation}\label{eq:fiberProduct}
    \{F(x_1;p),\ldots,F(x_j;p)\}=0,
\end{equation}
which is the $j^{\rm th}$ fiber product system for the projection $(x,p)\mapsto p$ \cite{FiberProducts}. We note that the isolated solutions to $f(x;p)$ are not
necessarily independent in the sense that imposing \eqref{eq:fiberProduct} for $j$ points may result in forcing more than $j$ endpoints to lie at infinity.

\subsection{Emergent solutions}
When a parameterized system has more equations than unknowns,  $f(x;p):\Cc^n\times\Cc^m\rightarrow\Cc^k$ with $k>n$, 
there may exist exceptional sets where the number of isolated solutions increases. A familiar example is a
linear system $Ax=b$ where full-rank matrix $A$ has more rows than columns. For most choices of $b$ in Euclidean space, the system is incompatible and has no solutions, but for $b$ lying in the column space of $A$, there will be a unique solution.
In the more general nonlinear case, one method for finding all isolated solutions is
to replace $f$ with a ``square'' randomization $R_n f:\Cc^n\times\Cc^m\rightarrow\Cc^n$ wherein
each of the $n$ polynomials of system $R_n f$ is a random linear combination of the polynomials
in $f$. Theorem~13.5.1 (item (2)) of \cite{SW05} implies that, with probability
one, the isolated points in $V(R_nf)$ include all the isolated points in $V(f)$.
After solving $R_n f$ by homotopy, one sorts solutions vs.\ nonsolutions by evaluating $f$ at each solution of $R_n f$.
If $V(R_n f)$ contains nonsolutions for generic parameters $p\in\Cc^m$, then there may exist an exceptional
set $P^*\subset\Cc^m$ where one or more of these satisfy $f$ to become solutions. 
Bertini's Theorem \cite[Thm~A.8.7]{SW05} tells us that the nonsolutions will be nonsingular with probability one.
However, if they 
emerge as singular solutions of $V(f)$ as $p\rightarrow P^*$, then they will be ill-conditioned near $P^*$.
Even if the nonsolutions remain
well-conditioned as solutions of $R_n f$, meaning that the Jacobian matrix
with respect to the variables
$J(R_n f)=R_n\cdot Jf$ is far from singular, 
the problem of solving $f$ for parameters in the vicinity of $P^*$ is ill-conditioned from the standpoint that
the number of solutions changes discontinuously as we approach $P^*$. The numerical difficulty arises in 
deciding whether or not $f(x;p)=0$ when the floating point evaluation of $f$ is near, but not exactly, zero.
For nonsingular emergent solutions, sensitivity analysis, e.g., singular value decomposition, of the full Jacobian matrix 
with respect to both the variables and the
parameters 
can estimate the distance in parameter space from the given parameters to $P^*$, while alpha theory or interval analysis 
can provide provable bounds. In the case of singular emergent solutions, multiplicity conditions will have to be
imposed as well (see below). 
In any case, the simultaneous emergence of $j$ solutions requires them to satisfy
the $j^{\rm th}$ fiber product system $\{f(x_1;p),\ldots,f(x_j;p)\}$.

\subsection{Sets of exceptional dimension} 
Polynomial systems often have solution sets of positive dimension. This happens by force if there are fewer
equations than unknowns, but it can happen more generally as well. Moreover, a polynomial system can have 
solution components at several different dimensions. For $x\in V(f)$, the \emph{local dimension} at $x$, denoted $\dim_x V(f)$,
is the highest dimension of all the solution components containing $x$. For a parameterized system with the natural projection $\pi(x,p) = p$,
the fiber above $p^*\in\Cc^m$ is $V(f(x;p^*))$ and the 
fiber dimension at point $(x^*,p^*)\in V(f(x,p))$
is $\dim_{x^*} V(f(x;p^*))$. Define $\sD_h$ as the closure of the set $\{(x,p)\in\Cc^N\times\Cc^m\suchthat \dim_x V(f(x;p))=h\}$,
which is an algebraic set. 
A parameterized polynomial system has a set of exceptional dimension wherever $\sD_H$ intersects $\sD_h$ for $H>h$, that is,
exceptions occur at parameter values $p^*\in\Cc^m$ where the fiber dimension increases. The sets $\pi(\sD_h)$ 
form a stratification of parameter space with each containment progressing to higher and higher fiber dimension.
Since the structure of the solution set changes every time there is a change in dimension, each such change is another example of ill-conditioning.
As presented in \cite{FiberProducts} and discussed below in Section~\ref{sec:HigherDim},
fiber products provide a way of finding exceptional sets.

In numerical algebraic geometry, sets of exceptional dimension can be understood as a case of emergent solutions.
Holding $p$ constant, a witness point set for a pure $D$-dimensional component of $V(f(x;p))$ is found by
intersecting with a codimension $D$ generic affine linear space, $L_D(x)$. For $D>n-k$, this is accomplished
by first computing the isolated solutions of the randomized system $\{R_{n-D}f(x;p),L_D(x)\}$, where $R_\ell f$ denotes~$\ell$ generic linear combinations of the polynomials in $f$ and $L_D(x)$ is a system of $D$ generic affine linear equations. When $p^*$ is on a set of exceptional dimension, nonsolutions emerge as solutions to~$f$ as $p\rightarrow p^*$. For a degree $d$ irreducible component to emerge, $d$ new witness points must emerge simultaneously, which leads
to a fiber product formulation of the same general form as \eqref{eq:fiberProduct}, with $F$ now defined as $F(x;p)=\{R_{n-D}f(x;p),L_D(x)\}$.
While the $d$ witness points of a degree $d$ component must satisfy the $d^{\rm th}$ fiber product,
it may happen that imposing the $j^{\rm th}$ fiber product for $j<d$ suffices. In particular, 
a different bound based on counting dimensions often comes into play first \cite{FiberProducts}.

\subsection{Exceptional decomposition}
Once one finds a witness set $\sW=\{f,L_D,W\}$ for a pure $D$-dimensional component $X$ of~$V(f)$, it is often of interest to decompose $X$ into its irreducible components, which are the 
closure of the connected components of $X$ after removing its singularities, $X\setminus X_{\rm sing}$. For a single polynomial, irreducible components correspond exactly with factors, so irreducible decomposition is the generalization of factorization to systems of polynomials. Ill-conditioning occurs near a point in parameter space where a component decomposes into more irreducible components than general points in the neighborhood. Again, we get a stratification of parameter space where components decompose more and more finely.

Every pure-dimensional algebraic set satisfies a linear trace condition, and irreducible components correspond with the smallest subsets of a witness point set $W$ that satisfy the trace test \cite{BHL,HR15,TraceLRS,SVWtraces}. Ill-conditioning occurs when a trace test for a proper subset $W_1\subsetneq W$ evaluates to approximately zero. We may then ask if there is a parameter point nearby where that test is exactly zero, indicating that $X$ decomposes, with $W_1$ representing a lower degree component. (The number of points in a witness set is equal to the degree of the algebraic set it represents.) 
As presented in Section~\ref{sec:Decompose}, since the trace
test involves all the points $W_1$ simultaneously, a kind of fiber product system ensues.

\subsection{Exceptional multiplicity}\label{sec:ExpMult}
Our opening example of a single polynomial with a double root generalizes to systems of polynomials. As we approach
a subset of parameter space where witness points merge, the components they represent coincide, forming a component of higher multiplicity.  
When we speak of the multiplicity of an irreducible component, we mean the multiplicity of its witness points cut out by a generic slice.
However, when randomization is used to find witness points of $V(f)$, $f(x):\Cc^n\rightarrow\Cc^k$ at dimensions $D>n-k$, the multiplicity of
a witness point as a solution of $\{R_{n-D}f(x),L_D(x)\}$ may be greater than its multiplicity as a solution of~$\{f,L_D(x)\}$, with equality only 
guaranteed for either multiplicity 1, that is, for nonsingular points
or when the multiplicity with respect to 
the randomized system is~$2$~\mbox{\cite[Thm.~13.5.1]{SW05}}. Section~\ref{sec:HigherMult} discusses in more detail how multiplicity is
defined in terms of Macaulay matrices
and local Hilbert functions. For the moment, it suffices to say that the Macaulay matrix evaluated at a generic point of an irreducible algebraic set 
reveals the set's local Hilbert function and multiplicity, and provides an algebraic condition for it. As such, we again get a stratification of parameter space associated with changing the local Hilbert function
and increasing the multiplicity.

Since every generic point of an irreducible algebraic set has the same multiplicity, the conditions necessary to set the multiplicity of a component
may be asserted for several witness points simultaneously. As in the previous cases, asserting an algebraic condition for several points simultaneously is
a form of fiber product.

\subsection{Summary}
Each case discussed above---divergent solutions, emergent solutions, exceptional dimension, exceptional decomposition, and exceptional multiplicity---can lead
to a kind of ill-conditioning wherein small changes in parameters result in a discrete change in an integer property of the solution set.
Given a parameterized polynomial system $f(x;p)$ along with parameters near such a discontinuity, one may 
consider variations in the parameters and pose the question of finding the nearest point in parameter space where
the exceptional condition occurs. In each case, imposing the exceptional condition on several points simultaneously results in a fiber product
system. In particular, when the exceptional condition applies to an irreducible component of degree $d>1$, 
it automatically applies to a set of $d$ witness points, and consequently, fiber products are key to robustifying numerical irreducible decomposition.

\section{Robustness framework}\label{sec:Robust}

As described in Section~\ref{sec:Background}, 
the key tool for robust numerical algebraic geometry
is {\em fiber products}~\cite{FiberProducts}.
In order to motivate the notation, 
we first consider a simple example of 
$f(x;p) = p_1x + p_2$.  Of course, for general $p\in\Cc^2$, 
$V(f(x;p))$ consists of a single point, namely $x=-p_2/p_1$.  
Suppose that one aims to compute a parameter point $p$
so that $V(f(x;p)) = \Cc$,
which we trivially know for this problem is simply $p=0$.
From the generic behavior of $f$, 
one knows that if $V(f(x;p))$ contains at least two distinct
points, then it must contain all of $\Cc$.  
Thus, for generic $c_1,c_2\in\Cc$, this can be formulated as 
the fiber product system
\begin{equation}\label{eq:SimpleFiberProduct}
\sF(x_1,x_2,p) = \left[\begin{array}{c} p_1 x_1 + p_2 \\ x_1 - c_1 \\ \hline 
p_1 x_2 + p_2 \\ x_2 - c_2 \end{array}\right] = 0.
\end{equation}
In $\sF$, the original parameters $p$ are variables 
and the original variables are copied twice to 
correspond with the two different
solutions.  For the projection map $\pi(x_1,x_2,p) = p$
onto the original parameters, $\pi(V(\sF)) = \{0\}$,
which, substituting back into the original $f$
shows that $V(f(x;0))=\Cc$ as requested.  

The basic idea is that each component system of the fiber product imposes a
condition on the parameters.  In \eqref{eq:SimpleFiberProduct}, 
the first component system cuts the original parameter space, a plane,
down to a line.  The second component system then cuts the line down to a point.
Any additional component systems would not reduce the dimension further
as the fiber over $p=0$ is $\Cc$.  The following formalizes this behavior.
To allow flexibility, we allow for auxiliary variables 
and constants arising from randomizations and slicing
to be included in each component
system.

\begin{theorem}\label{thm:BasicFiberProduct}
Suppose that $f(x;p)$ is a polynomial system and
$A\subset V(f(x,p))\subset\Cc^n\times\Cc^m$ is an irreducible algebraic set. 
For auxiliary variables $y$ and 
constants $c$ used for randomization and slicing, 
let $F_c(x,y,p)$ be a polynomial system which imposes a condition on
the parameter space when $(x,p)\in A$.  
For $a\geq0$ and generic constants $c_1,\dots,c_a$,
consider the fiber product system $\sF_a$ and projection map $\pi_a$ with
\begin{equation}\label{eq:FiberProdSystem}
\sF_a(x_1,\dots,x_a,y_1,\dots,y_a,p) = \left[\begin{array}{c} F_{c_1}(x_1,y_1,p) \\ \vdots \\ F_{c_a}(x_a,y_a,p)
\end{array}\right]
\hbox{~~~~and~~~~}
\pi_a(x_1,\dots,x_a,y_1,\dots,y_a,p) = p.
\end{equation}
Let $\Delta_a\subset V(\sF_a)$ be an algebraic set of components to ignore
such that there is a natural inclusion of $\Delta_a$ into $\Delta_{a'}$ for all $a<a'$.
Let $V_A(\sF_a)$ denote the solution set upon the restriction that $(x_j,p)\in A$
for $j=1,\dots,a$.  Then, exactly one of the following must hold:
\begin{enumerate}
    \item\label{Item:1basic} $\dim \overline{\pi_a(V_A(\sF_a)\setminus \Delta_a)} > \dim \overline{\pi_{a+1}(V_A(\sF_{a+1})\setminus\Delta_{a+1})}$ or
    \item\label{Item:2basic} $\dim \overline{\pi_a(V_A(\sF_a)\setminus\Delta_a)} = \dim \overline{\pi_{j}(V_A(\sF_j)\setminus\Delta_j)}$
for all $j\geq a$.
\end{enumerate}
\end{theorem}
\begin{proof}
Clearly, $\dim \overline{\pi_k(V_A(\sF_a)\setminus\Delta_a)} \geq \dim \overline{\pi_{a+1}(V_A(\sF_{a+1})\setminus\Delta_{a+1})}$ as the latter contains the same conditions as the former.
So, if Item~\ref{Item:1basic} does not hold, then we must have that
\hbox{$\dim \overline{\pi_a(V_A(\sF_a)\setminus\Delta_a)} = \dim \overline{\pi_{a+1}(V_A(\sF_{a+1})\setminus\Delta_{a+1})}$}.
Thus, the system $F_{c_{a+1}}$ did not cause the parameter space to drop in dimension.
By genericity, this must be true for any additional generic system and thus
Item~\ref{Item:2basic} holds.
\end{proof}

\begin{remark}
A common example for $A$ is to ignore base points as illustrated in Section~\ref{ex:Infinity}.
A common example for $\Delta_a$ to ignore are diagonal components,
e.g. 
$$\{(x_1,\dots,x_a,y_1,\dots,y_a,p)~|~(x_j,y_j) = (x_{j'},y_{j'})
\hbox{~for~} j \neq j'\}.$$
\end{remark}

The idea of Thm.~\ref{thm:BasicFiberProduct} is to keep imposing the same
condition until the dimension stabilizes.  However, one may want
to impose various conditions, such as conditions on different 
dimensions of the solution set.  This can be accomplished by stacking such fiber product systems.
The only difference is that the number of component systems used to stabilize the dimension
from the generic parameter space may be smaller on proper algebraic subsets of the parameter space.
An example of this is presented in Section~\ref{sec:6R}.

\begin{corollary}\label{Cor:FiberProduct}
With the basic setup from Thm.~\ref{thm:BasicFiberProduct}, suppose
that one aims to impose $r$ conditions on the parameter space,
where, for $b=1,\dots,r$ and $a\geq 0$, each $\sF^b_a$ is as in \eqref{eq:FiberProdSystem}.
Then, there exists $a_1,\dots,a_r\geq 0$ such that the dimension of the 
closure of the image of the solution set of
$$\sF = \left[\begin{array}{c} \sF^1_{a_1}(x_1,p) \\ \vdots \\ \sF^r_{a_r}(x_r,p) \end{array}\right]$$
after consistently removing components to ignore onto the original parameters $p$ stabilizes.  
\end{corollary}

\begin{remark}\label{remark:Lemma3Dimension}
In the multiplicity one case, Lemma~3 of \cite{WitnessProjection} provides a 
local linear algebra approach to compute the dimension of the image
from a general point on the component.  
From an approximation, one can utilize a numerical rank revealing method 
such as the singular value decomposition.  When all else fails,
one could use a guess and check method to determine if the fiber product
system described the proper parameter space.
\end{remark}

Once one has a properly constructed fiber product system $\sF$,
the next step is to recover a parameter value $p^*$ nearby $\hat{p}$,
where $\hat{p}$ is an approximation to an initial parameter value~$\tilde{p}$,
such that $p^*$ and $\tilde{p}$ are exceptional parameter values
lying in the projection of the solution set $V(\sF)$ onto the parameter space.
Here, one must choose a notion of ``nearby,'' 
such as the standard Euclidean distance or an alternative 
based on knowledge about uncertainty in~$\hat{p}$. 
Over the complex numbers, one may utilize
isotropic coordinates~\cite{Isotropic} so that the square of the 
Euclidean distance corresponds with a bilinear polynomial. 
To keep notation simple, we write this as the local optimization problem
\begin{equation}\label{eq:Optimization}
p^* = \arg \min \|p-\hat{p}\| \hbox{~such that~} \sF(x,p) = 0.
\end{equation}
Although there are many local optimization methods and distance metrics,
all examples below use the square of the standard Euclidean distance with
a gradient descent homotopy~\cite{GradDescent}.
In such cases, $\sF$ is constructed to be a well-constrained system 
and we aim to compute a nearby critical point of \eqref{eq:Optimization}
using a homogenized version of Lagrange multipliers:
$$\sG(x,p,\lambda) = \left[\begin{array}{c} \sF(x,p) \\
\lambda_0 \nabla(\|p-\hat{p}\|_2^2) + \sum_{j=1}^M \lambda_j \nabla(\sF_j) 
\end{array}\right]
$$
where $\nabla(q)$ is the gradient of $q$ and $\lambda\in\Pp^M$.
If $\hat{x}$ such that $\sF(\hat{x},\hat{p})\approx 0$, then the gradient descent homotopy
is simply
\begin{equation}\label{eq:GradDescentHomotopy}
\sH(x,p,\lambda,t) = \left[\begin{array}{c} \sF(x,p) - t \sF(\hat{x},\hat{p}) \\
\lambda_0 \nabla(\|p-\hat{p}\|_2^2) + \sum_{j=1}^M \lambda_j \nabla(\sF_j) 
\end{array}\right]
\end{equation}
where the starting point at $t=1$ is $(\hat{x},\hat{p},[1,0\dots,0])$.  
Note that such a gradient descent homotopy is local in that it may not 
work in cases such as when the perturbation is too large
or a ``nearby'' component did not actually exist with the given
formulation.  In such cases, one may need to consider
alternative formulations, e.g., isotropic coordinates,
as well as consider alternative local optimization methods.

\section{Projective space and solutions at infinity}\label{sec:FiniteRoots}

The first structure to consider applying this robust framework to is
to compute parameter values which have fewer finite solutions.  

\subsection{Solutions at infinity}\label{sec:SolAtInf}

For a parameterized polynomial system $f(x;p)$, one can consider
solutions at infinity by considering a homogenization (or multihomogenization) 
of $f$ with variables in projective space (or product of projective spaces).
Thus, solutions at infinity correspond with a homogenizing variable being equal to $0$.
For simplicity, suppose that we have replaced $f$ with a homogenized version
together with an affine linear patch to perform computations in affine space.
For a single condition in the spirit of Thm.~\ref{thm:BasicFiberProduct},
suppose that we are interested in reducing the number of finite solutions by
forcing solutions to be inside of the hyperplane at infinity defined by $x_0=0$.
This yields the following.

\begin{corollary}\label{co:SolAtInf}
With the setup described above, Thm.~\ref{thm:BasicFiberProduct} holds 
when applied to 
$$F(x,p) = \left[\begin{array}{c} f(x;p) \\ x_0 \end{array}\right].$$
In particular, this component system has no random constants nor auxiliary variables.
\end{corollary}

In the multiprojective setting, Cor.~\ref{Cor:FiberProduct} would apply to 
having solutions in different hyperplanes at infinity, e.g., see Section~\ref{sec:6R}.

For perturbed parameter values $\hat{p}$, one is looking 
for solutions to $f(x;\hat{p})=0$ for which a homogenizing coordinate is close to $0$.
If there are $s$ such points, then the number of component systems 
forming the fiber product is at most $s$ but could be strictly smaller
than~$s$ due to relations amongst the solutions. 

\subsection{Illustrative example}\label{ex:Infinity}

Consider the parameterized family of polynomial systems
\begin{equation}\label{eq:Infinity}
    f(x;p) = \left[\begin{array}{c}
    x_1^2 + p_1 x_1 + p_2 \\
    (x_1 +p_3) x_2 + 2 x_1 - 3
    \end{array}\right].
\end{equation}
For generic $p\in\Cc^3$, $f(x;p) = 0$ has two finite solutions. 
However, for the parameter values $\tilde{p}= (-2.3716, 0.98608803, -0.5377)$, taken as exact, $f$ has only one finite solution.  The reason for this reduction is that, for these parameter values, one of the two roots of the first polynomial happens to be $x_1=-\tilde{p}_3$, at which value
the second polynomial evaluates to $0x_2+2\tilde{p}_3-3\ne0$.
To demonstrate the robustness framework, we consider
starting with a perturbed parameter value $\hat{p}$
obtained from adding to $\tilde{p}$ an error in each coordinate
drawn from a Gaussian distribution with mean $0$ and standard deviation $0.01$, denoted $\mathcal{N}(0,0.01^2)$,
namely 
$\hat{p} = (-2.3728, 0.9607, -0.5349)$ to 4 decimal places.
If $\hat{p}$ is treated as exact, then 
$f(x;\hat{p})=0$ has two finite solutions
where one solution has large magnitude.  
So, we aim to recover $p^*$ near~$\hat{p}$ with one finite solution by pushing the large magnitude solution to~infinity.

The first step is to create a homogenization of $f$ in \eqref{eq:Infinity} together
with a generic affine patch.  Using a single homogenizing coordinate, say $x_0$, this yields
\begin{equation}\label{eq:InfinityHom}
    f(x;p) = \left[\begin{array}{c}
    x_1^2 + p_1 x_0 x_1  + p_2 x_0^2\\
    x_1 x_2 + 2 x_0 x_1 + p_3 x_0 x_2 - 3 x_0^2\\
    c_0 x_0 + c_1 x_1 + c_2 x_2 - 1
    \end{array}\right]
\end{equation}
where $c\in\Cc^3$ is chosen randomly.  
Note that, for every $p\in\Cc^3$, $x=(0,0,1/c_2)$ is a solution at infinity
so that we take $A = \overline{V(f)\setminus \left(\{(0,0,1/c_2)\}\times\Cc^3\right)}$
which is irreducible and consistent with Thm.~\ref{thm:BasicFiberProduct}.
For $\hat{p}$, numerical approximations of the solutions
are shown in Table~\ref{table:InfinityPoints}
with the first solution being the aforementioned point that is ignored.
The second solution listed has $x_0$ near $0$
and, thus, we aim to adjust the parameters so that 
it also lies on the hyperplane
at infinity defined by~$x_0=0$.
\begin{table}[ht]
\caption{Solutions to~\eqref{eq:InfinityHom} for $\hat{p}$ where $\text{i} = \sqrt{-1}$}
\label{table:InfinityPoints}
\centering
    \begin{tabular}[2in]{c r c l c r c l c r c l}
    \toprule
    Solution & \multicolumn{3}{c}{$x_0$} && \multicolumn{3}{c}{$x_1$} && \multicolumn{3}{c}{$x_2$} \\
    \midrule
    1 & 0.0000 & + & 0.0000i && 0.0000 &+& 0.0000i && --0.2235 &+& 0.8253i \\
    2 & 0.0020 & -- & 0.0072i && 0.0010 &--& 0.0037i && --0.2263 &+& 0.8289i \\
    3 & --0.0368 & + & 2.7475i && --0.0682 &+&5.0964i && 0.0198 &--& 1.4776i \\
    \bottomrule
    \end{tabular}
\end{table}

Since there is only a single solution to push to infinity, Cor.~\ref{co:SolAtInf} 
yields
\begin{equation}\label{eq:InfinityFiberSystem}
    \sF = \left[\begin{array}{c}
    x_1^2 + p_1 x_0 x_1  + p_2 x_0^2\\
    x_1 x_2 + 2 x_0 x_1 + p_3 x_0 x_2 - 3 x_0^2\\
    c_0 x_0 + c_1 x_1 + c_2 x_2 - 1 \\
    x_0
    \end{array}\right].
\end{equation}
Since there is only a single component system, there are no
other components to ignore so we take $\Delta = \emptyset$.
For illustration purposes, one can easily verify 
that the closure of the image of the projection onto $p\in\Cc^3$ 
of $V_A(\sF)$ is $V(p_3^2-p_1p_3+p_2)$.  
Such a defining equation can be determined using
symbolic computation, e.g., via Grobner bases,
or exactness recovery methods from numerical 
values, e.g., \cite{ExactnessRecovery}.

The critical point system constructed using homogenized Lagrange multipliers yields
\begin{equation}\label{eq:InfinityHomotopy}
    \mathcal{G} = \left[\begin{array}{c}
    x_1^2 + p_1 x_0 x_1  + p_2 x_0^2\\
    x_1 x_2 + 2 x_0 x_1 + p_3 x_0 x_2 - 3 x_0^2\\
    c_0 x_0 + c_1 x_1 + c_2 x_2 - 1 \\
    x_0\\
    \lambda_1(p_1 x_1 + 2 p_2 x_0) + \lambda_2(2x_1 + p_3x_2 - 6x_0) + \lambda_3 c_0 + \lambda_4\\
    \lambda_1 (2 x_1 + p_1 x_0) + \lambda_2 (x_2 + 2 x_0) + \lambda_3 c_1\\
    \lambda_2 (x_1 + p_3 x_0) + \lambda_3 c_2\\
    \lambda_0(p_1 - \hat{p}_1) + \lambda_1 x_0 x_1\\
    \lambda_0(p_2 - \hat{p}_2) + \lambda_1 x_0^2\\
    \lambda_0(p_3 - \hat{p}_3) + \lambda_2 x_0 x_2\\
    \end{array}\right].
\end{equation}
Taking the second solution in Table~\ref{table:InfinityPoints} as $\hat{x}$,
a gradient descent homotopy \eqref{eq:GradDescentHomotopy} recovers 
a nearby parameter value having the desired structure of only one finite solution,
which is provided in Table~\ref{table:InfinityParameters} to 8 decimal places.
The exceptional set is two-dimensional, so we do not expect to recover $\tilde{p}$ exactly,
just a point nearby consistent with the variance of the distribution of $\hat{p}$ around $\tilde{p}$.

\begin{table}[ht]
\caption{Initial (exact), perturbed (8 decimals), and recovered (8 decimals) parameter values}
\label{table:InfinityParameters}
\centering
    \begin{tabular}[2in]{cccc}
    \toprule
    Parameter & Initial ($\tilde{p}$) & Perturbed ($\hat{p}$) & Recovered ($p^*$) \\
    \midrule
    $p_1$ & --2.37160000  & --2.37284227  & --2.36891717  \\
    $p_2$ & \,\,\,0.98608803  & \,\,\,0.96067280  & \,\,\,0.96814820 \\
    $p_3$ & --0.53770000  & --0.53492792  & --0.52506952  \\
    \bottomrule
    \end{tabular}
\end{table}

Since the solution at infinity is singular, Remark~\ref{remark:Lemma3Dimension}
does not apply for computing dimensions using linear algebra.  
However, if we instead utilize a 2-homogeneous formulation, 
there are no base points.  Moreover, Remark~\ref{remark:Lemma3Dimension}
applies due to the nonsingularity.  
In particular, using two homogenizing coordinates, say $x_0$ and $x_3$, this yields
\begin{equation}\label{eq:Infinity2Hom}
    f(x;p) = \left[\begin{array}{c}
    x_1^2 + p_1 x_1 x_0 + p_2 x_0^2\\
    x_1 x_2 + 2 x_1 x_3 + p_3 x_0 x_2  - 3 x_0 x_3 \\
    \Box x_0 + \Box x_1 - 1\\
    \Box x_2 + \Box x_3 - 1\\
    \end{array}\right]
\end{equation}
where $\Box$ represents a random complex number.  Forming $\sF=\{f,x_3\}$, the
gradient descent homotopy leads to the same results 
as in Table~\ref{table:InfinityParameters}.

\section{Witness points and randomization}\label{sec:HigherDim}

The next structure to consider applying this robust framework to is
to compute parameter values which have solution components of various dimensions.

\subsection{Witness points}\label{sec:WitnessPoints}

As described in Section~\ref{sec:Background}, pure-dimensional solution components 
can be described by witness sets.  A key decision in numerical algebraic geometry,
such as part of a dimension-by-dimension algorithm
for computing a numerical irreducible decomposition, e.g., \cite{RegenCascade,Cascade,NumAlgGeom},
is to determine if a solution to a randomized system $Rf$ is actually a solution
to the original system $f$.  For exact systems, this can be determined
robustly by using the randomized system to refine the point and 
evaluate the original system using higher precision.  
Moreover, in the exact case, a nonsolution, i.e., 
a point satisfying $Rf=0$ and $f\neq0$, can be certifiably determined \cite{alphaCertified}.
This becomes a precarious task for systems with error.

Since discussions about multiplicity are provided later in Section~\ref{sec:HigherMult},
consider here that one is aiming for a pure $D$-dimensional component 
of degree $d$ where each irreducible component has multiplicity $1$
with respect to $f$.
This corresponds with a fiber product where
the component systems have $d$ solutions along various linear spaces
as summarized in~the~following.

\begin{corollary}\label{co:WitnessPoints}
With the setup described above, Thm.~\ref{thm:BasicFiberProduct} holds when applied to
$$F_c(x_1,\dots,x_d,p) = \left[\begin{array}{c} f(x_1;p) \\ L_c(x_1) 
\\ \vdots \\ f(x_d;p) \\ L_c(x_d) \end{array}\right]$$
where $c$ contains the coefficients of $L_c:\Cc^n\mapsto\Cc^D$.
\end{corollary}

For perturbed parameter values $\hat{p}$, one is looking 
for solutions to $Rf(x;\hat{p})=0$ for which~$f(x;\hat{p})$ evaluates to something close to $0$.
If one is considering witness points on components of various dimensions,
then Cor.~\ref{Cor:FiberProduct} applies with stacking fiber product systems
resulting from Cor.~\ref{co:WitnessPoints} for each dimension under consideration.  

\subsection{Illustrative example}\label{ex:WitnessPoints}

Consider the parameterized family of polynomial systems
\begin{equation}\label{eq:PosDim}
    f(x;p) = \left[\begin{array}{c}
    x_1 x_2 - 2x_1 + p_1 x_2 + p_2\\
    x_1^2 - 2x_1 + p_1x_1 + p_2
    \end{array}\right].
\end{equation}
For generic $p\in\Cc^3$, $f(x;p)=0$ consists of two isolated solutions.
However, for $\tilde{p}=(1,-2)$, the irreducible decomposition consists
of the line $V(x_1+1)$ and the point $(2,2)$ as shown 
in~Fig.~\ref{fig:PosDim_NID}.  Perturbing the parameters
with $\mathcal{N}(0,0.1^2)$ error
yielded $\hat{p}=(0.9876, -2.2542)$ to~4 decimal places
with the corresponding two isolated solutions also shown 
in~Fig.~\ref{fig:PosDim_NID}.
\begin{figure}[ht]
    \centering
\includegraphics[height=2.36in]{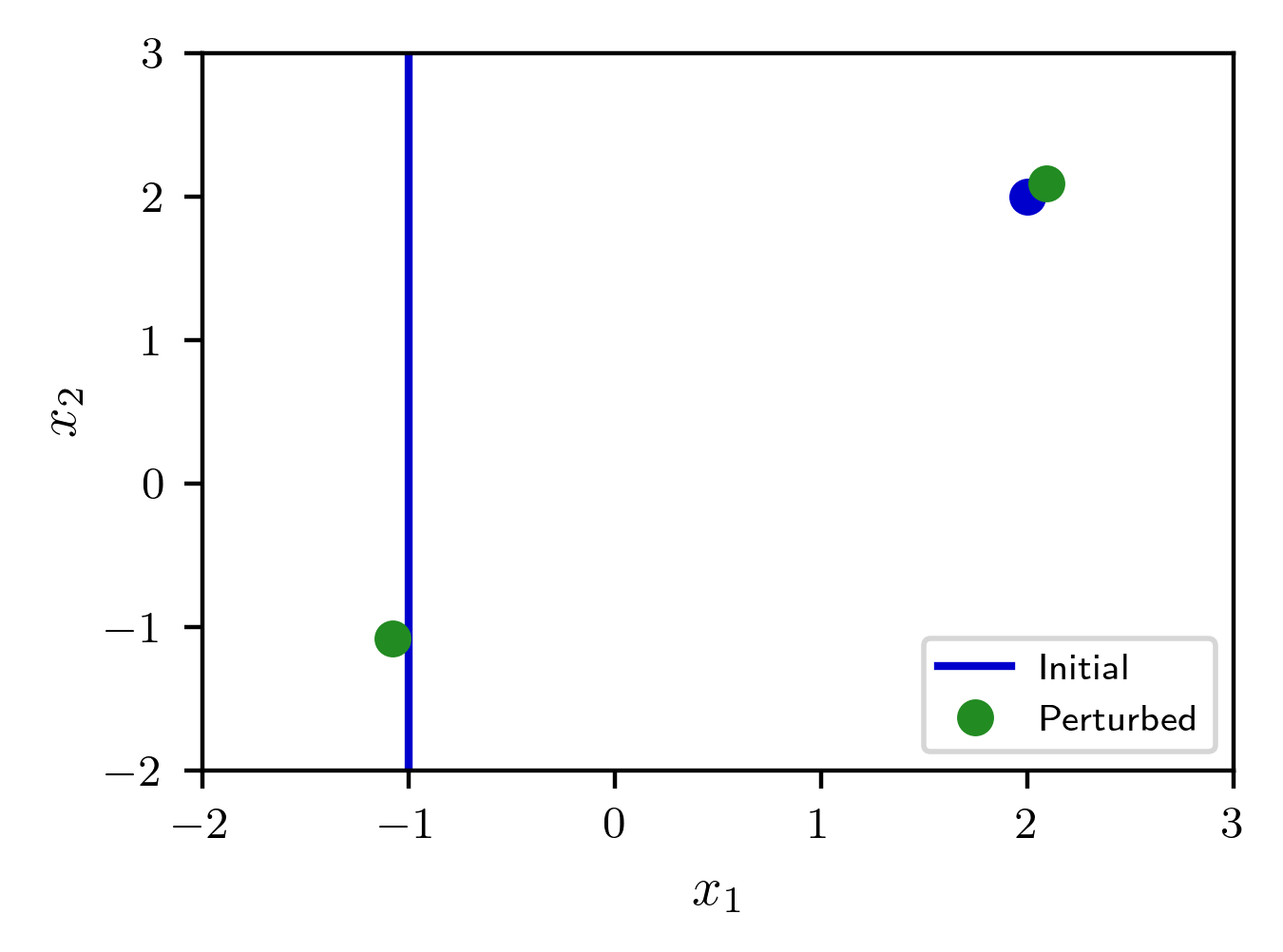}
    \caption{Solution sets for the initial and perturbed parameters}
    \label{fig:PosDim_NID}
\end{figure}

When considering witness points on a one-dimensional component,
the system under consideration has the form
\begin{equation}\label{eq:PosDim_Sliced}
f_R(x;p) = \left[
\begin{array}{c}
    Rf(x;p)\\
    L(x) 
\end{array}
\right] = \left[
\begin{array}{c}
     x_1 x_2 - 2x_1 + p_1 x_2 + p_2 + \Box (x_1^2 - 2x_1 + p_1 x_1 + p_2) \\
     \Box x_1 + \Box x_2 + \Box
\end{array}
\right]
\end{equation}
where each $\Box$ represents an independent random complex number
with $R = [1~\Box]$.  Solving~$f_R(x;\hat{p})$, 
there is one solution for which $f(x;\hat{p})$ 
is far from vanishing (called a nonsolution)
and one solution, call it $\hat{x}$ with~$\hat{x}_1$ in the vicinity of $-1$,
for which $f(\hat{x};\hat{p})$ is close to vanishing.  
Thus, we aim to recover
$p^*$ near~$\hat{p}$ for which this later point is an actual
witness point for a one-dimensional line.  
With this, the fiber product system becomes
\begin{equation}\label{eq:PosDim_FiberSystem}
    \mathcal{F} = \left[\begin{array}{c}
    x_1 x_2 - 2x_1 + p_1 x_2 + p_2\\
    x_1^2 - 2x_1 + p_1x_1 + p_2\\
    c_1 x_1 + c_2 x_2 + c_3
    \end{array}\right]
\end{equation}
and the critical point system using homogenized Lagrange multipliers
yields
\begin{equation}\label{eq:PosDim_Homotopy}
\mathcal{G} = \left[ \begin{array}{c}
    x_1 x_2 - 2x_1 + p_1 x_2 + p_2\\
    x_1^2 - 2x_1 + p_1 x_1 + p_2 \\
    c_1 x_1 + c_2 x_2 + c_3 \\ 
    \lambda_1 (x_2 - 2) + \lambda_2 (2 x_1 + p_1 - 2) + \lambda_3 c_1 \\
    \lambda_1 (x_1 + p_1) + \lambda_3 c_2 \\
    \lambda_0 (p_1 - \hat{p}_1) + \lambda_1 x_2 + \lambda_2 x_1 \\
    \lambda_0 (p_2 - \hat{p}_2) + \lambda_1 + \lambda_2\\
\end{array} \right].
\end{equation}
With $(\hat{x},\hat{p})$, tracking a single path with a gradient
descent homotopy \eqref{eq:GradDescentHomotopy} 
recovers a nearby parameter $p^*$ listed in Table \ref{table:PosDim_Parmaeters} 
to four decimal places.  Recomputing a numerical irreducible
decomposition for $f(x;p^*)$ yields a line and an isolated point
as requested.

\begin{table}[htb]
\caption{Initial (exact), perturbed (4 decimals), and recovered (4 decimals) parameter values}
\label{table:PosDim_Parmaeters}
\centering
    \begin{tabular}[2in]{cccc}
    \toprule
    Parameter & Initial ($\tilde{p}$) & Perturbed ($\hat{p}$) & Recovered ($p^*$)\\
    \midrule
    $p_1$ & \,\,\,1  & \,\,\,0.9876  & \,\,\,1.0992  \\
    $p_2$ & --2  & --2.2542  & --2.1984 \\
    \bottomrule
    \end{tabular}
\end{table}

We repeated this process by sampling 500 points from a bivariate Gaussian distribution centered at the initial parameter values $\tilde{p} = (1, -2)$ with covariance matrix $\Sigma = 0.1^2 I_2$ where each sample represents 
parameter values with error. 
In Fig.~\ref{fig:PosDim_ProjectedPoints}, the aforementioned
$\tilde{p}$ is shown as a square, and $\hat{p}$ and $p^*$ are triangles,
while the additional sampled values and recovered parameters are shown as circles.  For this simple problem, it is easy to verify
that all recovered parameter values lie along the line $V(2p_1+p_2)$.

\begin{figure}[htb]
    \centering
    \begin{tabular}{cc}
    \includegraphics[scale=0.88]{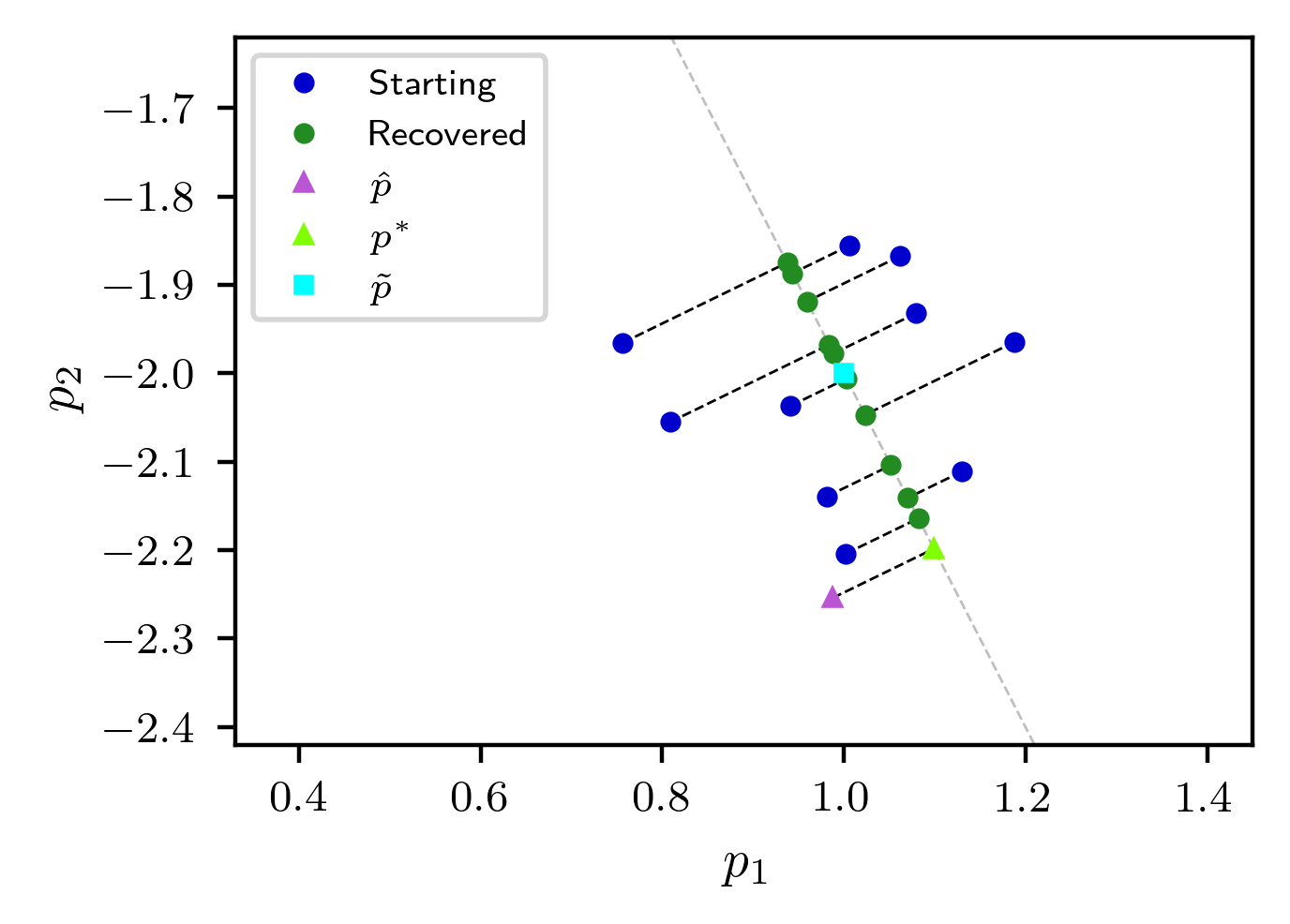} & 
    \includegraphics[scale=0.88]{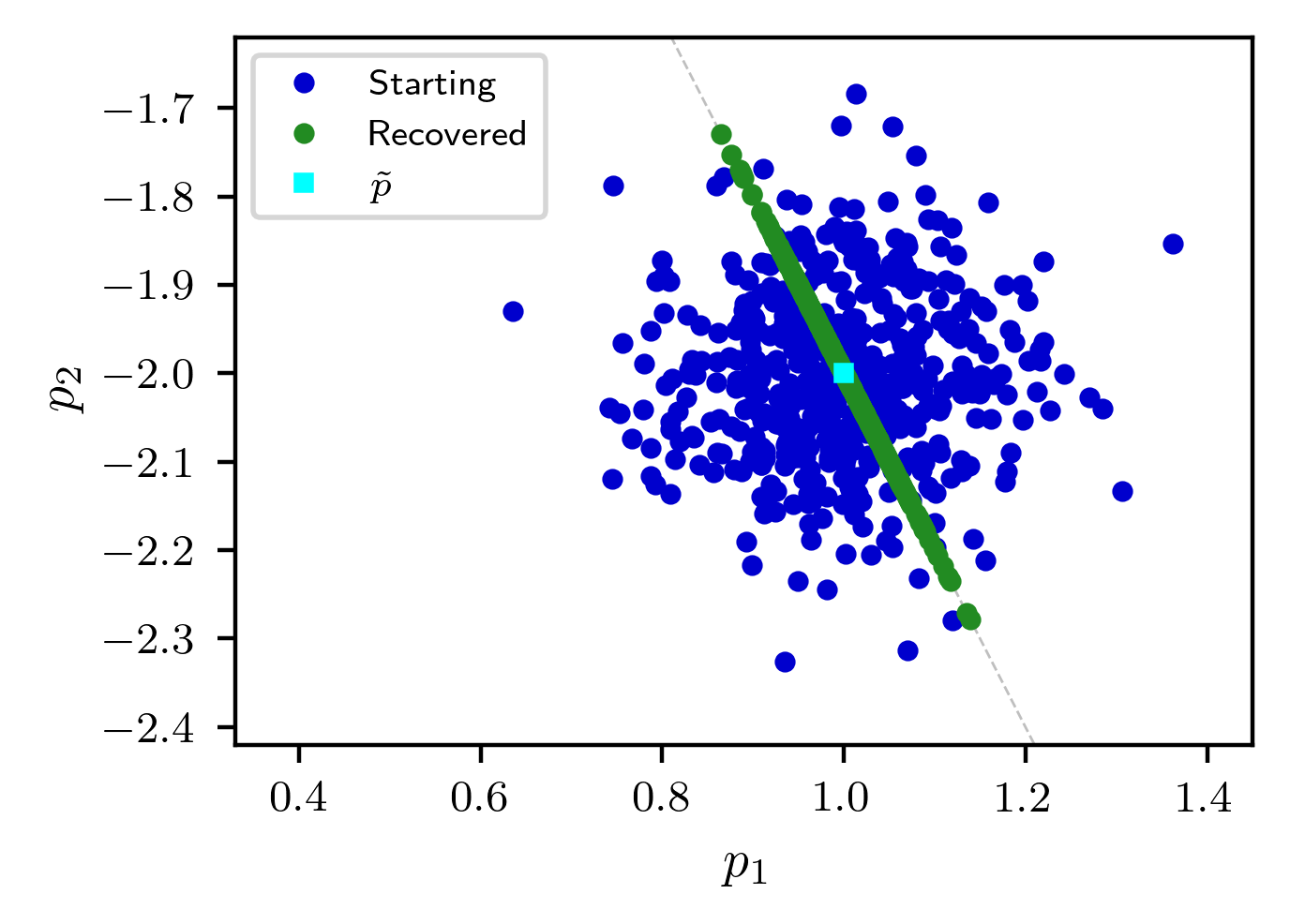} \\
    (a) & (b) 
    \end{tabular}
    \caption{(a) Illustration of recovering parameters for various
    perturbations including the example summarized in Table~\ref{table:PosDim_Parmaeters};
    (b) Illustration using 500 samples}
    \label{fig:PosDim_ProjectedPoints}
\end{figure}

To visualize marginal histograms of the recovered parameter values from 
the $500$ samples, Fig.~\ref{fig:PostDim_Histograms} shows
the $p_1$ and $p_2$ coordinates along with an intrinsic
coordinate parameterizing the line with $0$ corresponding
to $\tilde{p}=(1,-2)$.  
\begin{figure}[htb]
    \centering
    \begin{tabular}{ccc}
    \includegraphics[scale=0.62]{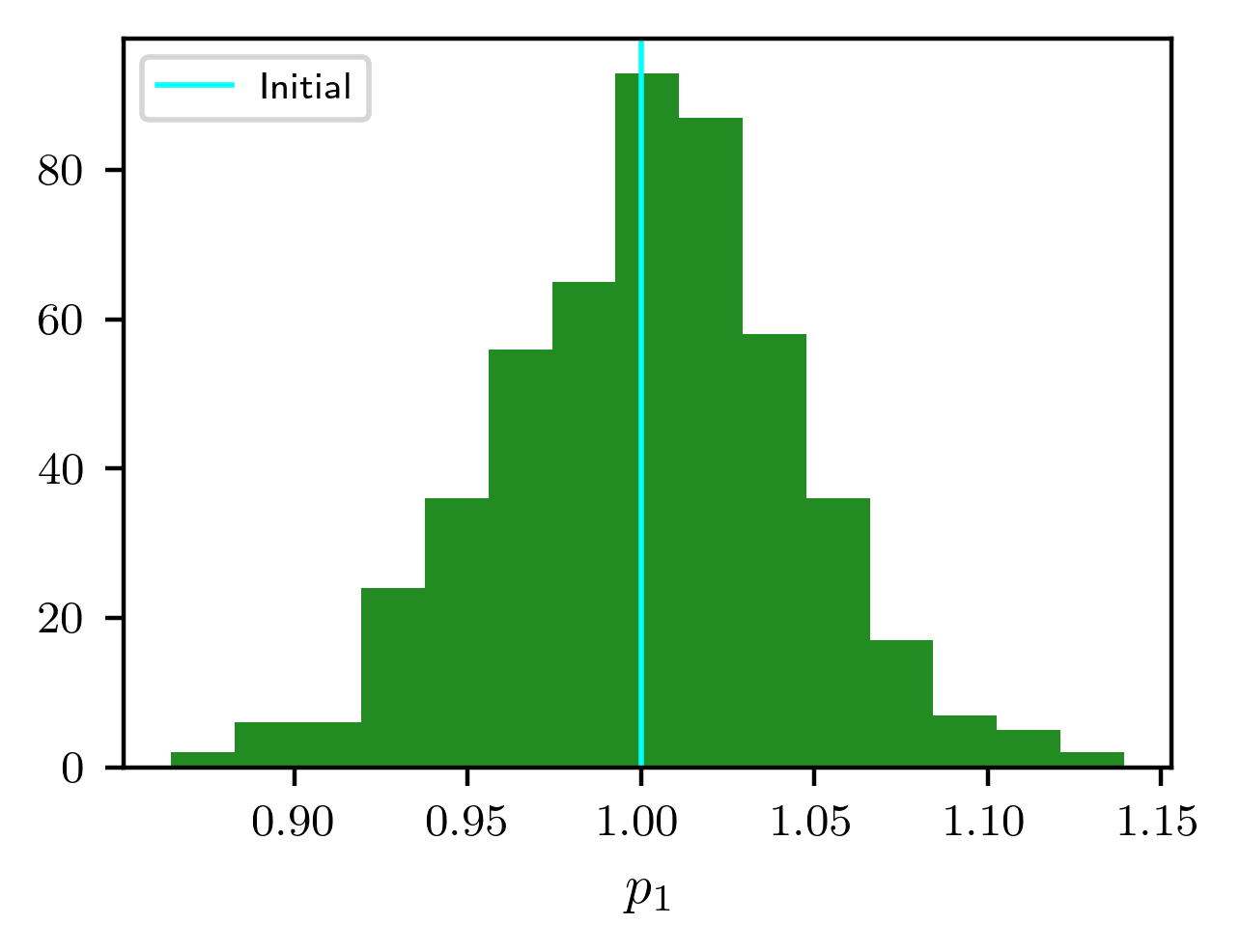} & 
    \includegraphics[scale=0.62]{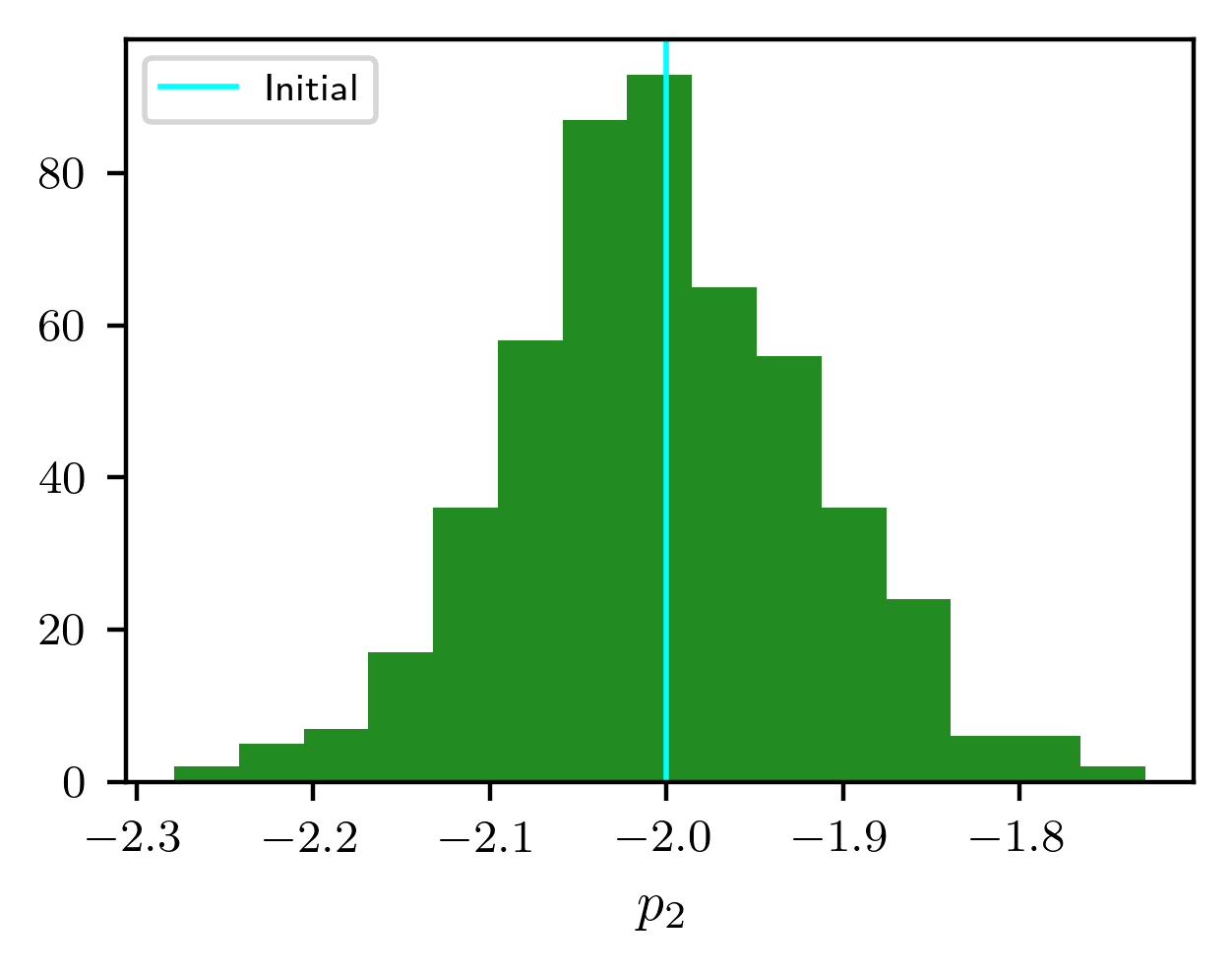} & 
    \includegraphics[scale=0.62]{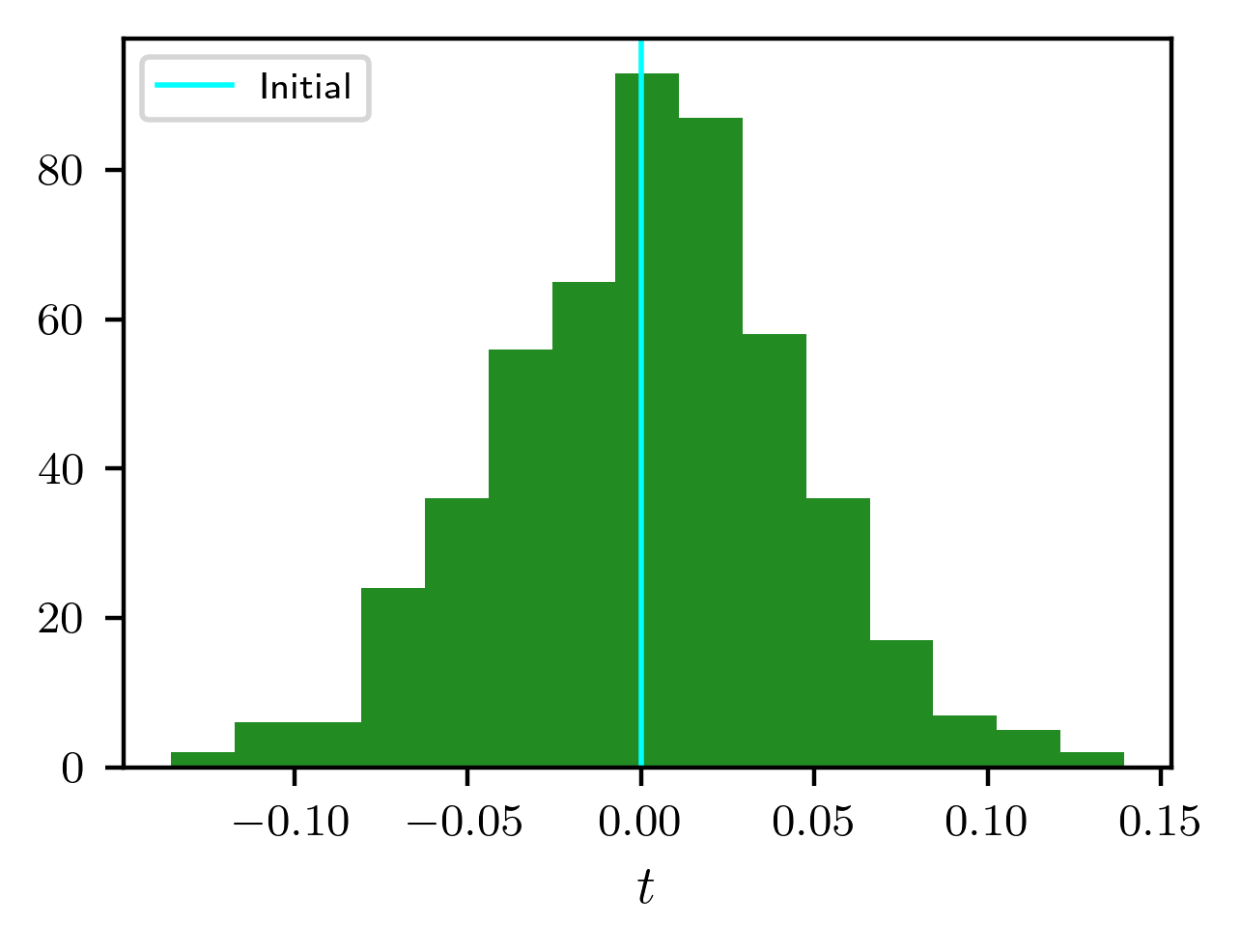} \\
    (a) & (b) & (c) 
    \end{tabular}
    \caption{Histograms for (a) $p_1$, (b) $p_2$, and (c) intrinsic parameterizing coordinate for recovered parameter values from $500$ samples}
    \label{fig:PostDim_Histograms}
\end{figure}

For a standard multivariate Gaussian, all marginals are Gaussian.  So, if we orthogonally project a multivariate Gaussian centered at $\tilde{p}$ onto a linear space passing through $\tilde{p}$, this will yield a Gaussian distribution in the linear space. 
In the
case presented here, the perturbations from the initial parameters are both generated as zero-mean with standard deviation $0.1$, so 
the recovered parameters along the line should be centered on the initial parameters with that same standard
deviation. Figure~\ref{fig:PostDim_Histograms}(c) is consistent with that expectation. 
Moreover, orthogonally projecting the distribution of the perturbed parameters onto the line perpendicular to $V(2p_1+p_2)$
will also be distributed as Gaussian with standard deviation~$0.1$. 
If one were given just the perturbed parameters and their accuracy, described
as a statistical distribution, one could calculate a confidence 
in the null hypothesis that the given
parameters are drawn from a distribution centered on an initial value for which $V(f(x;\tilde{p}))$ has one component that is a line.
In this case of a single remaining degree of freedom in parameter space, a $Z$-score for $||\hat p-p^*||$ would be
informative. 

We will not delve into statistical analyses for more general cases. Nevertheless,
we remark that
if the exceptional set in parameter space is codimension $s$
and the incoming parameters are perturbed
from the exceptional set with a normal distribution 
$\sN(0,\sigma^2 I)$, then the squared distance 
$\sigma^{-2}||\hat p-p^*||^2$
is a chi-squared distribution with $s$ degrees of freedom. (This assumes that the exceptional set
is locally smooth and $\sigma$ is small enough that a local linearization of the exceptional set
is accurate on the scale of $\sigma$.) 
If the perturbations have a more general normal distribution, say $\sN(0,\Sigma)$,
then it would be appropriate to change the norm used in~\eqref{eq:Optimization} to $(p-\hat p)^T\Sigma^{-1}(p-\hat p)$
so that we are searching for a maximum likelihood estimate. The same norm would then enter into a 
chi-square confidence estimate.

\section{Traces and numerical irreducible decomposition}\label{sec:Decompose}

With Section~\ref{sec:HigherDim} considering witness points,
the next structure to consider applying this robust framework to is
computing parameter values which have solution components 
that decompose into various irreducible components.

\subsection{Reducibility}\label{sec:Irreducible}

In a numerical irreducible decomposition, the collection
of witness points is partitioned into subsets corresponding
with the irreducible components.  One approach for performing
this decomposition is via the trace test \cite{BHL,HR15,TraceLRS,SVWtraces}
and a key decision is to determine when the linear trace vanishes.  
For exact systems, this can be determined robustly 
by computing the linear trace to higher precision, but becomes
an uncertain task for systems with error as perturbations
tend to destroy reducibility.  

The form of the trace test that we will employ here is the
second derivative trace test from \cite{BHL}
as this can be employed locally.  Suppose that 
$\{f,L_D,W\}$ is a witness set for a pure $D$-dimensional
component $X$ of $V(f)$.  Since discussions about multiplicity
are provided in Section~\ref{sec:Multiplicity},
suppose that each irreducible component of $X$ has multiplicity
$1$ with respect to~$f$.
Moreover, by replacing $f$ with a randomization,
we can assume that $f:\Cc^n\rightarrow\Cc^{n-D}$.  
Let $W_r \subset W$ consist of $r$ points.  
Then, there is a pure $D$-dimensional
component $X'\subset X$ with $X'\cap V(L_D) = W_r$
if and only if, for a general $L_D':\Cc^n\rightarrow\Cc^D$ 
and general $\alpha\in \Cc^n$,
\begin{equation}\label{eq:Trace}
\alpha\cdot \sum_{j=1}^r \ddot{w}_j = 0
\end{equation}
where $\{w_1,\dots,w_r\} = X'\cap V(L_D')$, and
$\dot{w}_j$ and $\ddot{w}_j$ satisfy
$$
\left[\begin{array}{c} Jf(w_j) \\ JL_D'(w_j) \end{array}\right] \cdot
\dot{w}_j = \left[\begin{array}{c} 0 \\ 1 \end{array}\right]
\hbox{~~and~~}
\left[\begin{array}{c} Jf(w_j) \\ JL_D'(w_j) \end{array}\right] \cdot \ddot{w}_j = -\left[\begin{array}{c} \dot{w}_j^T\cdot \Hessian(f_1)(w_j) \cdot \dot{w}_j \\ \vdots 
\\ \dot{w}_j^T\cdot \Hessian(f_{n-D})(w_j) \cdot \dot{w}_j \\ 0
\end{array}\right].
$$
Thus, for chosen slices and randomizing vectors, 
one can utilize fiber products to recover reducibility
of a component of degree $r$.

\begin{corollary}\label{co:Irreducible}
With the setup described above, Thm.~\ref{thm:BasicFiberProduct}
holds when applied to 
$$F_c(x_1,\dot{x}_1,\ddot{x}_1,\dots,x_r,\dot{x}_r,\ddot{x}_r,
p) = \left[\begin{array}{c} f(x_1;p) \\ L_c(x_1) \\
\left[\begin{array}{c} Jf(x_1;p) \\ JL_c(x_1) \end{array}\right]\cdot \dot{x}_1 - \left[\begin{array}{c} 0 \\ 1 \end{array}\right] \\
\left[\begin{array}{c} Jf(x_1;p) \\ JL_c(x_1) \end{array}\right] \cdot \ddot{x}_1 + \left[\begin{array}{c} \dot{x}_1^T\cdot \Hessian(f_1)(x_1;p) \cdot \dot{x}_1 \\ \vdots 
\\ \dot{x}_1^T\cdot \Hessian(f_{n-D})(x_1;p) \cdot \dot{x}_j \\ 0
\end{array}\right] \\ \vdots \\ f(x_r;p) \\ L_c(x_r) \\
\left[\begin{array}{c} Jf(x_r;p) \\ JL_c(x_r) \end{array}\right]\cdot \dot{x}_r - \left[\begin{array}{c} 0 \\ 1 \end{array}\right] \\
\left[\begin{array}{c} Jf(x_r;p) \\ JL_c(x_r) \end{array}\right] \cdot \ddot{x}_r + \left[\begin{array}{c} \dot{x}_r^T\cdot \Hessian(f_1)(x_r;p) \cdot \dot{x}_r \\ \vdots 
\\ \dot{x}_r^T\cdot \Hessian(f_{n-D})(x_r;p) \cdot \dot{x}_r \\ 0
\end{array}\right] \\ 
\alpha_c\cdot(\ddot{x}_1 + \cdots + \ddot{x}_r)
\end{array}\right]$$
where $c$ contains the coefficients of $L_c:\Cc^n\rightarrow\Cc^D$
and $\alpha_c\in\Cc^n$.
\end{corollary}

For perturbed parameter values $\hat{p}$, one is looking for
collections of points for which~\eqref{eq:Trace} is close to $0$.
The situation in Cor.~\ref{co:Irreducible} covers the case
when $X$ has degree $d$ and one is considering decomposing $X$
into a degree $r$ and degree $d-r$ component (where \hbox{$1\leq r\leq d-r \leq d$}). If one is considering factorization into more than two components
or the factorization of components in different dimensions, then 
Cor.~\ref{Cor:FiberProduct} applies with stacking fiber product
systems resulting from Cor.~\ref{co:Irreducible}.
Moreover, in the multiplicity $1$ case considered here, 
\mbox{Remark~\ref{remark:Lemma3Dimension}~applies}.

\subsection{Illustrative example}\label{ex:Irreducible}

Consider the parameterized family of polynomial systems
\begin{equation}\label{eq:Zeke}
    f(x;p) = p_1 + p_2 x_1 + p_3 x_1^2 + p_4 x_1^3 + p_5 x_1 x_2 + p_6 x_1^2 x_2 + p_7 x_1^3 x_2 + p_8 x_2^2 + p_9 x_1^2 x_2^2 + p_{10} x_1 x_2^3
\end{equation}
from \cite{wu2017numerical}.  
For generic $p\in\Cc^{10}$, $f(x;p) = 0$ defines 
a quartic plane curve. 
The problem described in~\cite[Ex.~2]{wu2017numerical} considers
the parameters $\tilde{p}=(-30,20,18,-12,12,-8,0,-5,3,2)$
with perturbation $\hat{p} = (-30,20,18,-12,12.000007,-8,0.0000003,-5,3,2)$
so that 
$$
    f(x;\tilde{p}) = (3x_1^2 + 2x_1 x_2 - 5)(x_2^2 - 4x_1 + 6)
    \hbox{~~and~~}
    f(x;\hat{p}) = f(x;\tilde{p}) + 0.0000003 x_1^3 x_2 + 0.000007 x_1 x_2
$$
with the corresponding quartic plane curves illustrated in Fig.~\ref{fig:Zeke_NID}.
\begin{figure}[htb]
    \centering
    \begin{tabular}{cc}
    \includegraphics[scale=0.88]{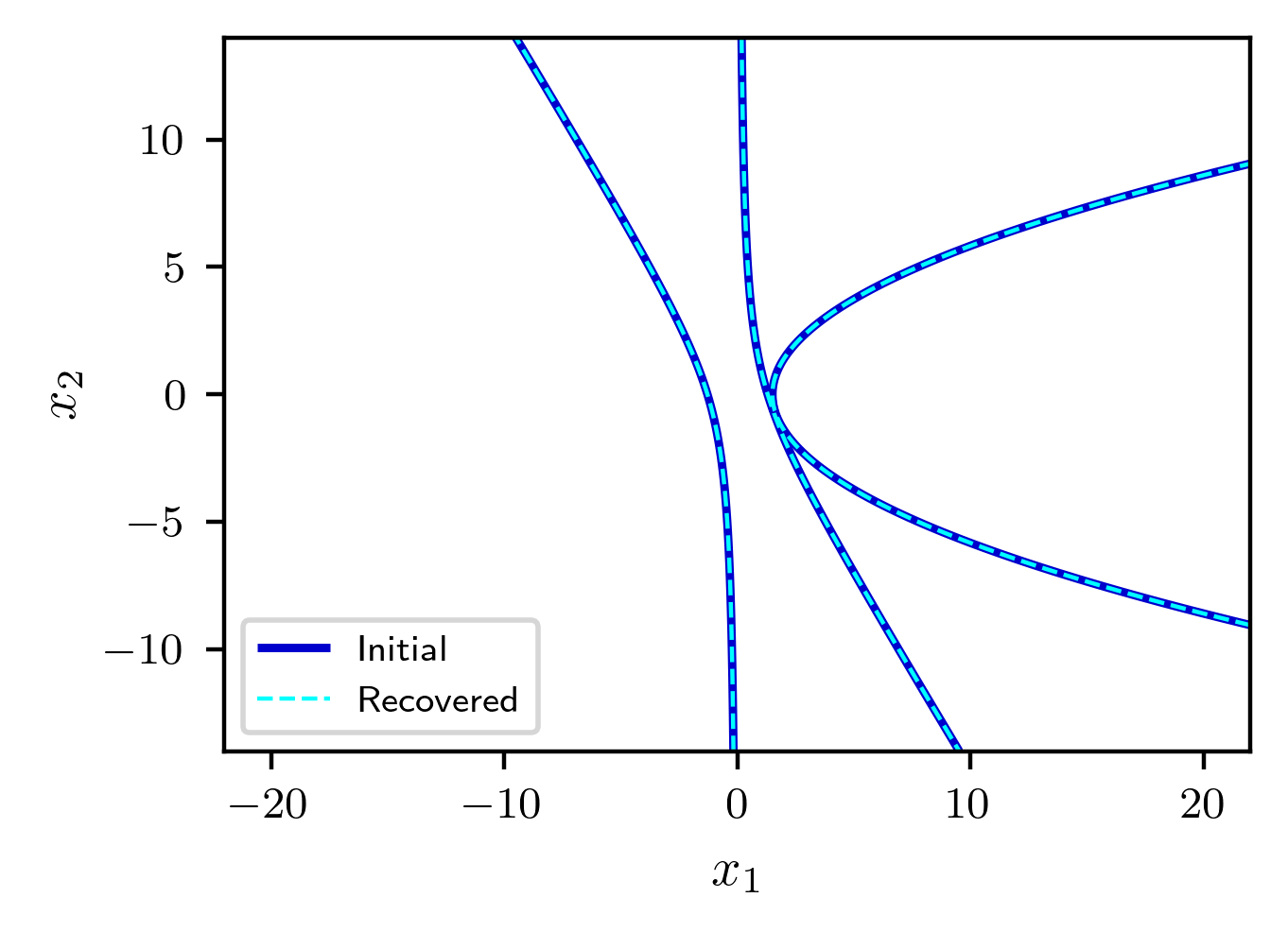} &
    \includegraphics[scale=0.88]{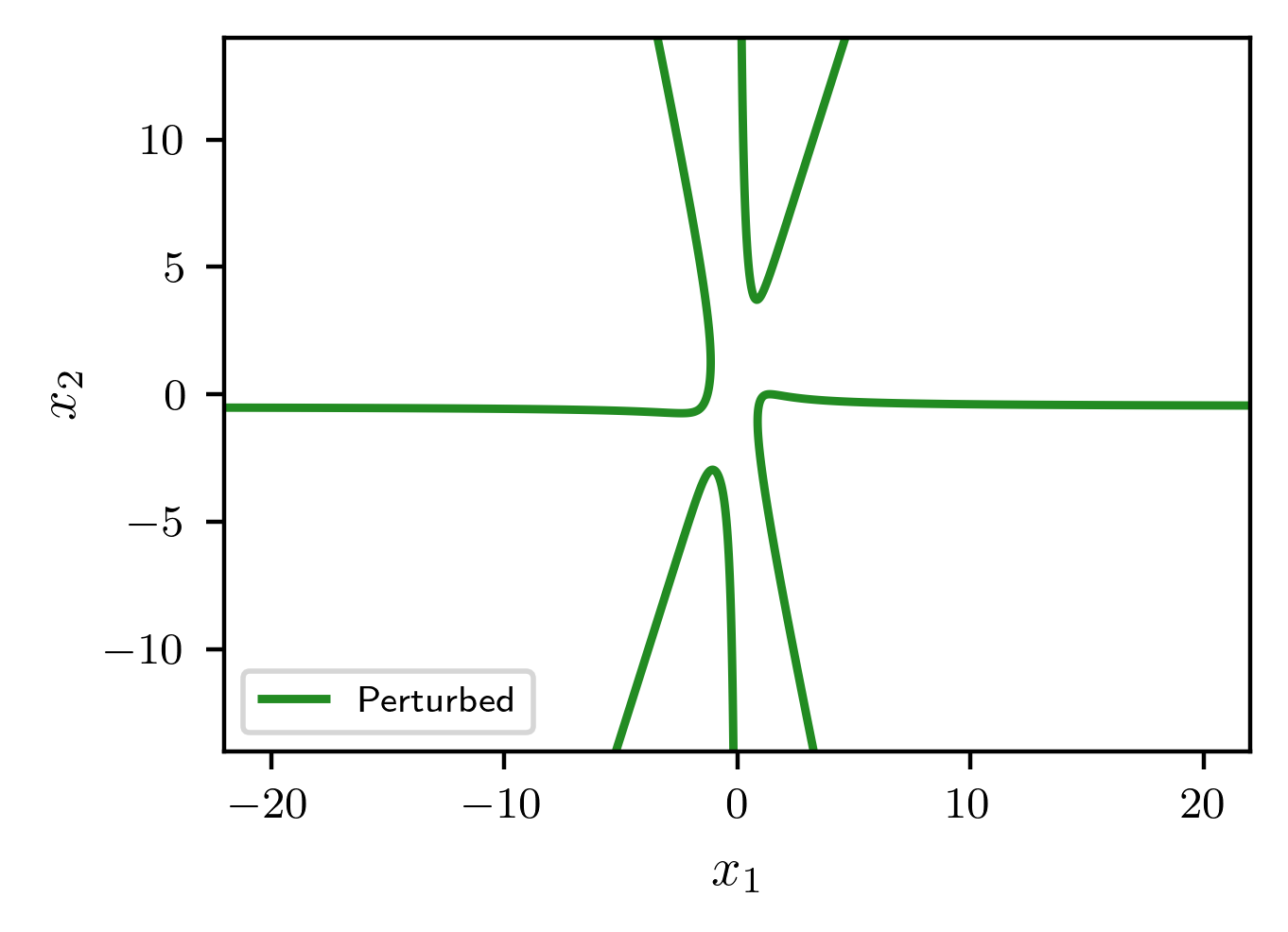} \\
    (a) & (b) \\
    \end{tabular}
    \caption{Solution sets corresponding to the (a) initial and recovered parameters, and (b)~perturbed parameters}
    \label{fig:Zeke_NID}
\end{figure}

For illustration, Table~\ref{tab:Trace}
considers the result of intersecting $f(x;\hat{p})=0$ 
with the linear space defined by $2x_1 - 3x_2 = 1$.
Clearly, we see that both $\ddot{w}_1+\ddot{w}_2$
and $\ddot{w}_3+\ddot{w}_4$ are close to zero 
indicating that we should consider computing parameters $p^*$
near $\hat{p}$ for which $f(x;p^*)$
factors into two quadratics via the second derivative trace test,
i.e., apply Cor~\ref{co:Irreducible} with $r=2$.

\begin{table}[htb]
\caption{Summary of solutions (8 decimals) satisfying $f(x;\hat{p}) = 2x_1-3x_2-1=0$}
\label{tab:Trace}
\centering
    \begin{tabular}[2in]{crrr}
    \toprule
    $j$ & \multicolumn{1}{c}{$w_j$} & \multicolumn{1}{c}{$\dot{w_j}$} & \multicolumn{1}{c}{$\ddot{w}_j$}\\
    \midrule
    $1$ & $\begin{array}{r} 
    1.15384590 \\ 0.43589727 \end{array}$ & 
    $\begin{array}{r} 
    0.08241763\\ -0.27838824 \end{array}$ & 
    $\begin{array}{r}
    0.00546655 \\ 0.00364437 \end{array}$ \\[0.15in]
    $2$ & $\begin{array}{r} 
    -0.99999993 \\ -0.99999995\end{array}$ & 
    $\begin{array}{r} 
    0.07142858 \\ -0.28571428 \end{array}$ & 
    $\begin{array}{r}
    -0.00546648 \\-0.00364432\end{array}$ \\[0.15in]
    $3$ & $\begin{array}{r} 
    1.64589862 \\ 0.76393241\end{array}$ & 
    $\begin{array}{r} 
    -0.17082057 \\ -0.44721371 \end{array}$ &
    $\begin{array}{r}
    0.13416408 \\ 0.08944272\end{array}$ \\[0.15in]
    $4$ & $\begin{array}{r} 
    8.35410056 \\ 5.23606704\end{array}$ & 
    $\begin{array}{r} 
    1.17082044  \\ 0.44721363 \end{array}$ & 
    $\begin{array}{r}
    -0.13416415 \\ -0.08944277\end{array}$ \\
    \bottomrule
    \end{tabular}
\end{table}

Utilizing Remark~\ref{remark:Lemma3Dimension}, Table~\ref{table:Zeke_FiberProducts} shows the corresponding 
dimension of parameter space based on the number of component systems
utilized.  In this case, one sees a stabilization of the dimension
of the parameter space to $6$ using $4$ component systems.
That is, we expect to recover parameters $p^*$ contained in a
$6$-dimensional parameter space.  
The resulting gradient descent homotopy \eqref{eq:GradDescentHomotopy} 
with homogenized Lagrange multipliers is a square system
consisting of the same number of variables and equations which is 
also listed in Table~\ref{table:Zeke_FiberProducts}.  
Finally, Table~\ref{table:Zeke_FiberProducts} also records
the numerical irreducible decomposition of $f(x;p^*)$ for
the corresponding recovered $p^*$.  This column also indicates 
that $4$ component systems are needed
with the recovered parameters (to $7$ decimal places) provided
in Table~\ref{table:Zeke_Parmaeters}.  
For comparison, written using double precision, the recovered factorization from \cite{wu2017numerical} and $p^*$ is provided in Table~\ref{table:Zeke_Factors}.  The key difference is
that \cite{wu2017numerical} enforced $p_7=0$
so that the third factor
maintained the same monomial structure as the exact system 
while the second derivative trace test 
only enforced factorability.  Hence, with the additional constraint,
the recovered factorization in~\cite{wu2017numerical}
is further away ($3.76\cdot10^{-6}$) from $p^*$ than $\hat{p}$ ($3.15\cdot10^{-6}$).  
Of course, one could impose the additional parameter space
condition, namely, $p_7 = 0$, 
with the second derivative trace test approach and recover
the same factorization as~\cite{wu2017numerical}.

\begin{table}[htb]
\caption{Summary for different component systems}
\label{table:Zeke_FiberProducts}
\centering
    \begin{tabular}[2in]{cccr}
    \toprule
    Component Systems & Dimension & \multicolumn{1}{c}{System Size} & \multicolumn{1}{c}{Recovered Components}\\
    \midrule
    1 & 9 & \,\,\,35 & One component of degree 4\\
    2 & 8 & \,\,\,60 & One component of degree 4\\
    3 & 7 & \,\,\,85 & One component of degree 4\\
    4 & 6 & 110 & \,\,Two components of degree 2\\
    5 & 6 & 135 & \,\,Two components of degree 2\\
    \bottomrule
    \end{tabular}
\end{table}

\begin{table}[htb]
\caption{Initial (exact), perturbed (exact), and recovered (7 decimals) parameter values}
\label{table:Zeke_Parmaeters}
\centering
    \begin{tabular}[2in]{cccc}
    \toprule
    Parameter & Initial ($\tilde{p}$) & Perturbed ($\hat{p}$) & Recovered ($p^*$)\\
    \midrule
    $p_1$ & --30  & --30.0000000  & --30.0000003  \\
    $p_2$ & \,\,\,20  & \,\,\,20.0000000  & \,\,\,19.9999994 \\
    $p_3$ & \,\,\,18  & \,\,\,18.0000000  & \,\,\,18.0000003 \\
    $p_4$ & --12  & --12.0000000  & --11.9999997 \\
    $p_5$ & \,\,\,12  & \,\,\,12.0000070  & \,\,\,12.0000057 \\
    $p_6$ & \,\,\,--8  & \,\,\,--8.0000000 & \,\,\,--8.0000019 \\
    $p_7$ & \,\,\,\,\,\,0  & \,\,\,\,\,\,0.0000003  & \,\,\,--0.0000014 \\
    $p_8$ & \,\,\,--5  & \,\,\,--5.0000000  & \,\,\,--5.0000002 \\
    $p_9$ & \,\,\,\,\,\,3  & \,\,\,\,\,\,3.0000000  & \,\,\,\,\,\,2.9999992 \\
    \,\,$p_{10}$ & \,\,\,\,\,\,2  & \,\,\,\,\,\,2.0000000  & \,\,\,\,\,\,2.0000006 \\
    \bottomrule
    \end{tabular}
\end{table}

\begin{table}[htb]
\caption{Comparison of factors (double precision)}
\label{table:Zeke_Factors}
\centering
    \begin{tabular}[2in]{ll}
    \toprule
\multirow{3}{*}{Factors from~\cite{wu2017numerical}} & 
\hbox{\footnotesize $-30.0000005908641$} \\
& \hbox{\footnotesize $1-0.600000000000000 x_1^2-0.400000158490569 x_1 x_2$} \\

& \hbox{\footnotesize $1-0.666666625730982 x_1+0.166666656432746 x_2^2$} \\
\hline
\multirow{3}{*}{Factors from~$p^*$} & \hbox{\footnotesize $-30.0000003490653$} \\
& \hbox{\footnotesize $1 - 0.600000003568289 x_1^2 - 0.400000104966183 x_1 x_2$} \\
& \hbox{\footnotesize $1 - 0.666666639555509 x_1 + 0.166666672330951 x_2^2  - 7.87997843076905 \cdot 10^{-8} x_1 x_2$} \\
    \bottomrule
    \end{tabular}
\end{table}

\section{Multiplicity and local Hilbert function}\label{sec:HigherMult}

The final structure we consider for applying this robust framework
to is to compute parameter values which have solutions 
with specified multiplicity and local Hilbert function.  

\subsection{Macaulay matrix}\label{sec:Multiplicity}

For a univariate polynomial $u(x)$, a number $x^*$
is said to have multiplicity $\mu\geq 0$ if and only if 
$u(x^*) = u'(x^*) = u''(x^*) = \cdots = u^{(\mu-1)}(x^*) = 0$
and $u^{(\mu)}(x^*)\neq0$.
For a multivariate polynomial system, derivatives are replaced
with partial derivatives leading to different ways of having a solution
with multiplicity $\mu$.  One approach for computing multiplicity
in multivariate systems is via Macaulay matrices first introduced
in \cite{Macaulay16} and utilized in various methods such as
\cite{LDT,DZ05,GHPS,MultAlg,Stetter04,MechMobility,Closedness09} to name a few.

For $\alpha\in\Zz_{\geq0}^n$, define
$$|\alpha| = \alpha_1+\cdots+\alpha_n, ~~~\alpha! = \alpha_1!\cdots\alpha_n!, ~~~\hbox{and}~~~ \partial_\alpha = \frac{1}{\alpha!}\frac{\partial^{|\alpha|}}{\partial x^\alpha}.$$
For $x^*\in\Cc^n$, consider the 
linear functional $\partial_\alpha[x^*]$ from polynomials in $x$ 
to $\Cc$ defined by
$$\partial_\alpha[x^*](g) = (\partial_\alpha g)(x^*)$$
which is simply the coefficient of $(x-x^*)^\alpha$ 
in an expansion of $g(x)$ about $x^*$.  
For a polynomial system $f:\Cc^n\rightarrow\Cc^k$ and $d\in\Zz_{\geq0}$,
the $d^{\rm th}$ Macaulay matrix of $f$ at $x^*$ is
$$
M_d(f,x^*) = \left[ \partial_\alpha[x^*]\left((x-x^*)^\beta f_j\right)
\hbox{~~such that~~} |\alpha|\leq d, |\beta|\leq\max\{0,d-1\}, j = 1,\dots,k
\right]
$$
where the rows are indexed by $(\beta,j)$ while the columns are indexed
by $\alpha$.  Define $\beta\leq\alpha$ if $\beta_a\leq\alpha_a$ for all
$a=1,\dots,n$.  By Leibniz rule, 
$$
\partial_\alpha[x^*]\left((x-x^*)^\beta f_j\right) = 
\left\{\begin{array}{ll}
\partial_{\alpha-\beta}[x^*](f_j) & \hbox{~if~}\beta \leq \alpha, \\
0 & \hbox{~otherwise.}\end{array}\right.
$$
For example, $M_0(f,x^*) = f(x^*)$ and $M_1(f,x^*) = \left[f(x^*) ~~ Jf(x^*)\right]$.  Moreover, there are matrices $A_d(f,x^*)$ and $B_d(f,x^*)$
such that 
\begin{equation}\label{eq:MMatrix}
M_{d+1}(f,x^*) = \left[\begin{array}{cc} M_d(f,x^*) & A_d(f,x^*) \\
0 & B_d(f,x^*) \end{array}\right].
\end{equation}

The local Hilbert function of $f$ at $x^*$ is
$$h_{f,x^*}(d) = \dim \NULL M_d(f,x^*) - \dim \NULL M_{d-1}(f,x^*)$$
where one defines $\dim \NULL M_{-1}(f,x^*) = 0$.
In particular, $x^*\in V(f)$ if and only if $h(0) = 1$.
Moreover, if $x^*\in V(f)$, then $x^*$ is isolated in~$V(f)$ if and only if 
there exists $d^*\geq0$ such that $h_{f,x^*}(d) = 0$ for all $d>d^*$
with multiplicity $\mu = \dim \NULL M_{d^*}(f,x^*) = \sum_{d=0}^{d^*} h_{f,x^*}(d)$.
Such a statement was used in \cite{LDT,MechMobility} to construct
a local approach to decide if $x^*\in V(f)$ was isolated or not. 

The following shows how to impose a local Hilbert function condition.

\begin{corollary}\label{co:Multiplicity}
With the setup described above, let $d\in\Zz_{\geq0}$ and $x^*\in\Cc^n$.
If $h_{f,x^*}(0)=1$ and $h_{f,x^*}(j)\in\Zz_{\geq1}$ for $j=1,\dots,d$,
then $f(x^*) = 0$ and, for generic square unitary matrices $R_1,\dots,R_d$
of appropriate sizes, the following collection
of linear systems 
defined in terms of matrices $\Sigma_j$ and $\Lambda_j$
has a unique solution for all $j=1,\dots,d$:
\begin{equation}\label{eq:LocalHilbert}
\left[\begin{array}{cc}
M_{j-1}(f,x^*) & A_{j-1}(f,x^*) \\
0 & B_{j-1}(f,x^*)
\end{array}\right]\cdot \left[\begin{array}{c}
\Sigma_{j} \\
R_j\left[\begin{array}{c} \Lambda_j \\ I_{h(j)} \end{array}\right]
\end{array}\right]=0
\end{equation}
where $h(j) = h_{f,x^*}(j)$ and $I_a$ is the $a\times a$ identity matrix.
\end{corollary}
\begin{proof}
The result follows immediately from \eqref{eq:MMatrix}, 
the definition of the local Hilbert function,
and \cite[Thm.~2]{RankDef}.  In particular,
the matrix $R_j\left[\begin{array}{c} \Lambda_j \\ I_{h(j)} \end{array}\right]$ has rank $h(j) = h_{f,x^*}(j)$
and thus \eqref{eq:LocalHilbert} yields
$h(j)$ independent null vectors that are not 
contained in the null space of $M_{j-1}(f,x^*)$ as required.
Uniqueness follows from~\cite[Thm.~2]{RankDef} where the 
total number
of unknowns in $\Sigma_j$ and~$\Lambda_j$ is precisely the dimension of the corresponding 
Grassmannian.
\end{proof}

In order to enforce various local Hilbert functions separately at several points, Cor.~\ref{co:Multiplicity} can be applied 
individually for each point and then all such systems can be collected together.
To enforce multiplicity of a component, one applies a local Hilbert
function condition at each of the witness points separately
with respect to the polynomial system and slicing system together.  
Then, one simply takes fiber products resulting from the component systems
in Cor.~\ref{co:Multiplicity} stacked together via Thm.~\ref{thm:BasicFiberProduct} and Cor.~\ref{Cor:FiberProduct}.
Additionally, if one aims to enforce a Hilbert function of a zero-scheme, this approach can naturally be generalized following~\cite{GHPS}.

For perturbed parameter values $\hat{p}$, one is looking for
solutions to $f(x;\hat{p})=0$ for which the corresponding
Macaulay matrices are nearly rank deficient.
This can be determined using numerical rank revealing methods
such as the singular value decomposition to determine appropriate
null space conditions to apply.  

\subsection{Illustrative example}\label{ex:Multiplicity}

Similar to Section \ref{ex:WitnessPoints}, computing a numerical irreducible decomposition of
\begin{equation}\label{eq:f_multiplicity}
f(x;p) = \left[
\begin{array}{c}
     f_1\\
     f_2
\end{array}
\right]=
\left[
\begin{array}{c}
     x_1^3 - 2 p_1 x_1^2 - 2 x_1^2 + p_1^2 x_1 + 4 p_1 x_1 - p_1^2 - p_2 \\
     x_1^2 x_2 - 2 x_1^2 - 2 p_1 x_1 x_2 + 4 p_1 x_1 + p_1^2 x_2 - p_1^2 - p_2
\end{array}
\right],
\end{equation}
with $\tilde{p} = (1,1)$, has a line and an isolated point in the solution space. However, for this system, the line has multiplicity two.
When these parameters are perturbed, say with $\mathcal{N}(0, 0.1^2)$ error yielding $\hat{p} = (1.2346, 1.0089)$ to 4 decimal places, the point and line structure breaks into three isolated points. In this case, we want to recover nearby parameters giving the special structure of a one-dimensional line with multiplicity two and an isolated point. As in Section \ref{ex:WitnessPoints}, we randomize to a single equation and add a slice. After solving
\begin{equation}\label{eq:Multiplicity_Sliced}
f_R(x;p) = \left[
\begin{array}{c}
    Rf(x;p)\\
    L(x)
\end{array}
\right] = \left[
\begin{array}{c}
     f_1 + \Box f_2 \\
     \Box x_1 + \Box x_2 + \Box
\end{array}
\right] = 0,
\end{equation}
we choose one of the two solutions near the one-dimensional line. To recover the multiplicity at this component, we add the condition outlined in Cor.~\ref{co:Multiplicity}.  In particular,
as mentioned in Section~\ref{sec:ExpMult},
since the randomized system having multiplicity 2 implies
the original system has multiplicity 2, we simply work
with the randomized system with \eqref{eq:LocalHilbert}
corresponding to
\begin{equation}\label{eq:JacobianMultEx}
Jf_R(x;p)\cdot R_1\cdot \left[\begin{array}{c} \lambda \\1   \end{array}\right] = 0.
\end{equation}
Hence, $\sF$ consists of the polynomials in \eqref{eq:Multiplicity_Sliced} and \eqref{eq:JacobianMultEx}.  
Using a gradient descent
homotopy~\eqref{eq:GradDescentHomotopy} with homogenized
Lagrange multipliers yields the recovered parameters listed
in Table~\ref{table:Multiplicity_Parmaeters}
and pictorially represented in Fig.~\ref{fig:points_multiplicity}.

Similar to Section~\ref{ex:WitnessPoints},
we repeated this process with 500 samples from a bivariate Gaussian distribution centered at the initial parameter values $\tilde{p}=(1,1)$ with 
covariance matrix $\Sigma = 0.1^2 I_2$.
The results of this experiment are summarized in Fig.~\ref{fig:points_multiplicity}.  
For this simple problem, it is easy to verify that all
recovered parameter values lie along the parabola
$V(p_1^2-p_2)$.  Figure~\ref{fig:histogram_multiplicity}
shows histograms of the marginal distributions 
for $p_1$, $p_2$, and along an intrinsic parameterization
of the tangent line to the parabola at $\tilde{p}$.  

\begin{table}[htb]
\caption{Initial (exact), perturbed (4 decimals), and recovered (4 decimals) parameter values}
\label{table:Multiplicity_Parmaeters}
\centering
    \begin{tabular}[2in]{cccc}
    \toprule
    Parameter & Initial ($\tilde{p}$)& Perturbed ($\hat{p}$)& Recovered ($p^*$)\\
    \midrule
    $p_1$ & 1  &  1.2346 &  1.0479 \\
    $p_2$ & 1  &  1.0089 &  1.0980 \\
    \bottomrule
    \end{tabular}
\end{table}

\begin{figure}[htb]
    \centering
    \begin{tabular}{cc}
    \includegraphics[scale=0.88]{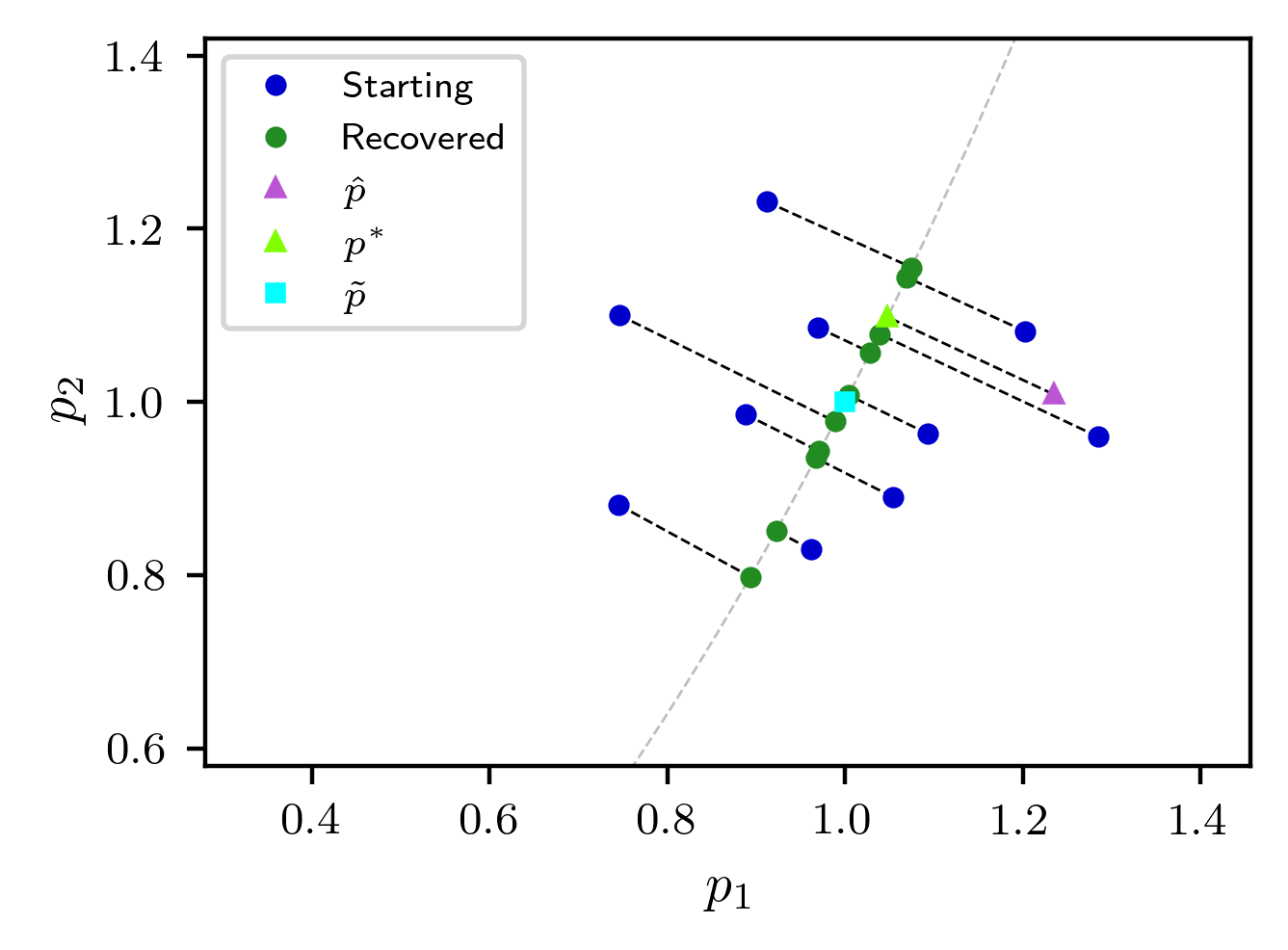} &
    \includegraphics[scale=0.88]{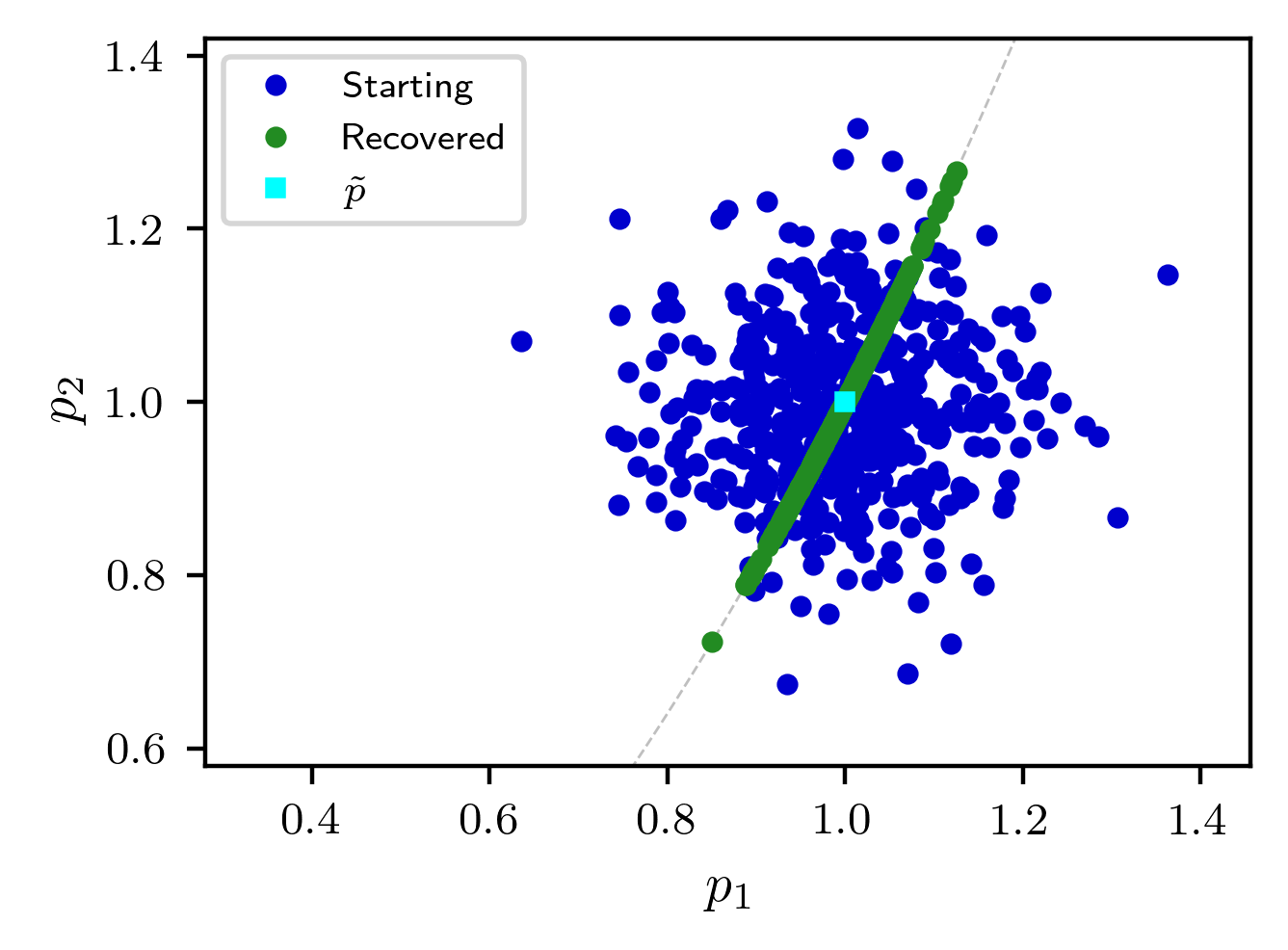}\\
    (a) & (b) \\
    \end{tabular}
    \caption{(a) Illustration of recovering parameters for various
    perturbations including the example summarized in Table~\ref{table:Multiplicity_Parmaeters};
    (b) Illustration using 500 samples}
    \label{fig:points_multiplicity}
\end{figure}

\begin{figure}[htb]
    \centering
    \begin{tabular}{ccc}
    \includegraphics[scale=0.62]{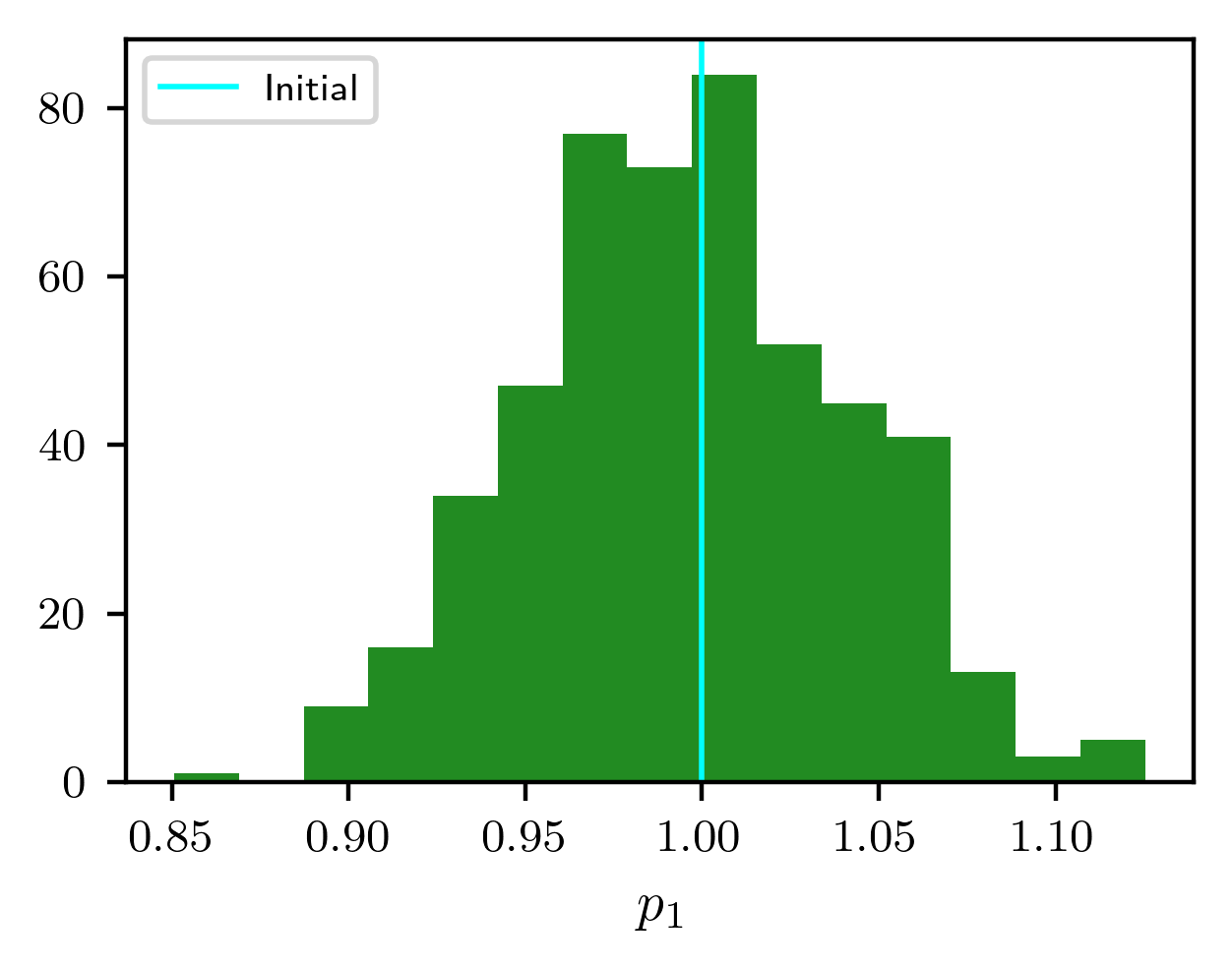} &
    \includegraphics[scale=0.62]{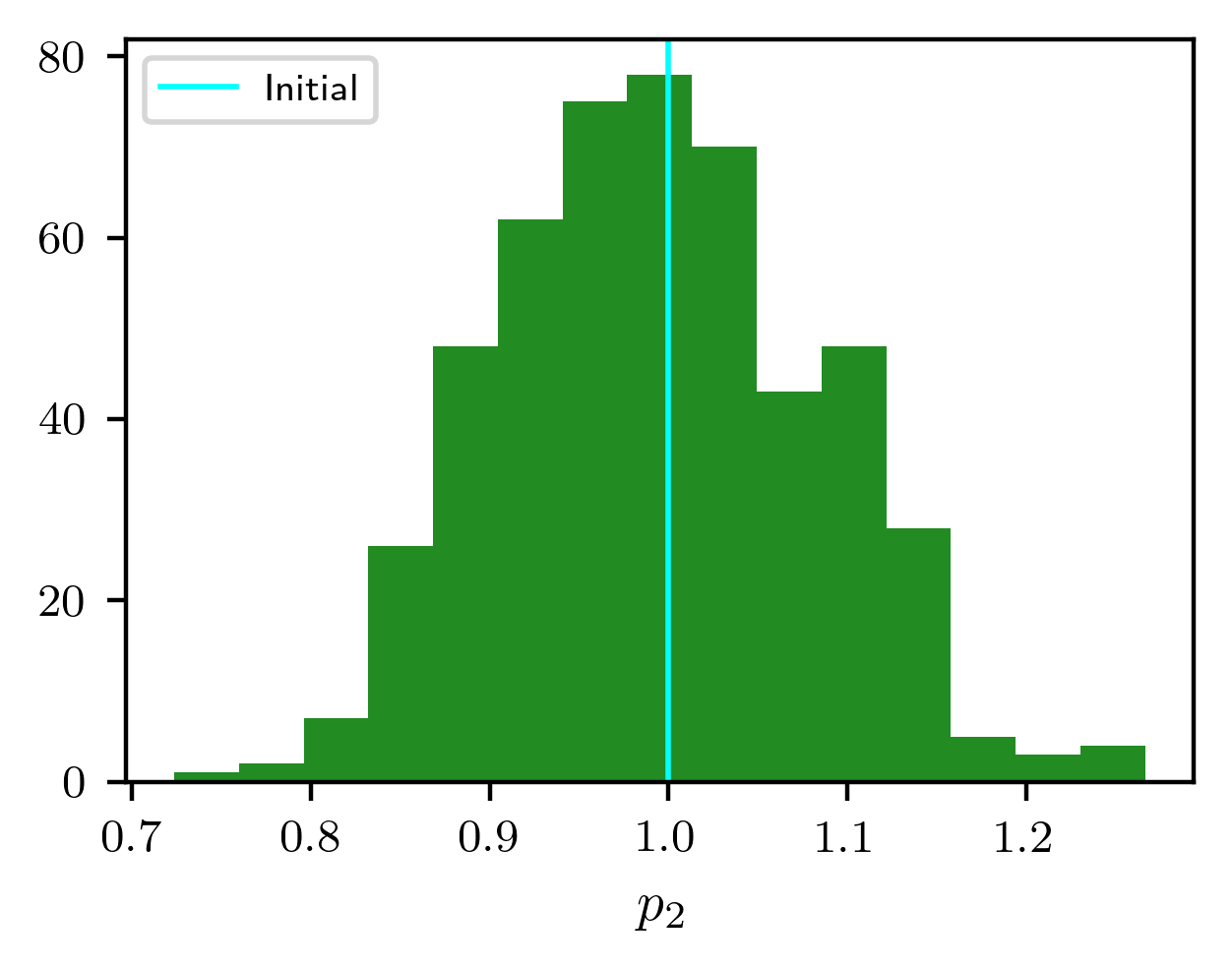} & \includegraphics[scale=0.62]{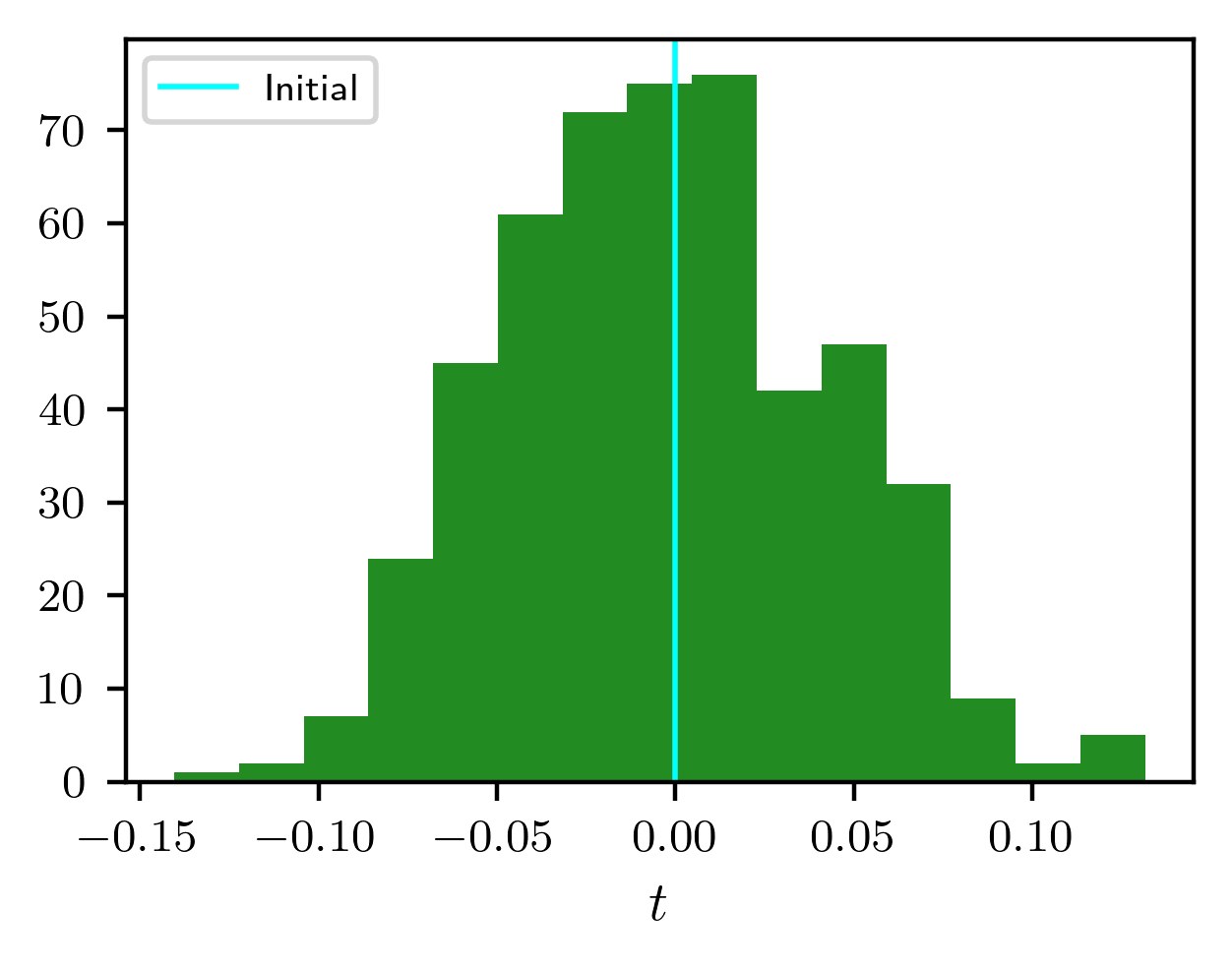} \\
    (a) & (b) & (c) \\
    \end{tabular}
    \caption{Histograms for (a) $p_1$, (b) $p_2$, and (c) intrinsic parameterizing coordinate along tangent line for recovered parameter values from $500$ samples}
    \label{fig:histogram_multiplicity}
\end{figure}

\section{Kinematic examples}\label{sec:Examples}
The examples in Sections \ref{sec:FiniteRoots}--\ref{sec:HigherMult} were 
designed for illustrative purposes.  The following
considers three examples derived
from the field of kinematics.

\subsection{Decomposable 4-bar coupler curve}\label{sec:4bar}
Consider the 4-bar linkage given by the parameterized family of polynomial systems
\begin{equation}\label{eq:4Bar}
    f(x;p) = \left[\begin{array}{c}
    x_1^2 + x_2^2 - p_1^2\\
    (x_3-p_2)^2 + x_4^2 - p_3^2\\
    (x_1-x_3)^2 + (x_2-x_4)^2 - p_4^2
    \end{array}\right].
\end{equation}
For generic $p\in\Cc^4$, $V(f(x;p))$ is an irreducible sextic $(d=6)$ curve. It is known that $p_1=p_3$ and $p_2=p_4$
yields a parallelogram linkage and the solution set factors into a quadratic and quartic curve. 
Considering initial parameter values $\tilde{p}=(1,2,1,2)$, 
we perturbed the parameters with $\mathcal{N}(0, 0.01^2)$ error 
yielding $\hat{p} = (1.0025, 2.0101, 1.0098, 2.0014)$ 
rounded to four decimals. 
For illustration, Fig.~\ref{fig:4Bar} shows the solution set 
projected into $(x_1, x_4)$ space.
\begin{figure}[!ht]
    \centering
    \begin{tabular}{cc} 
    \includegraphics[scale=0.88]{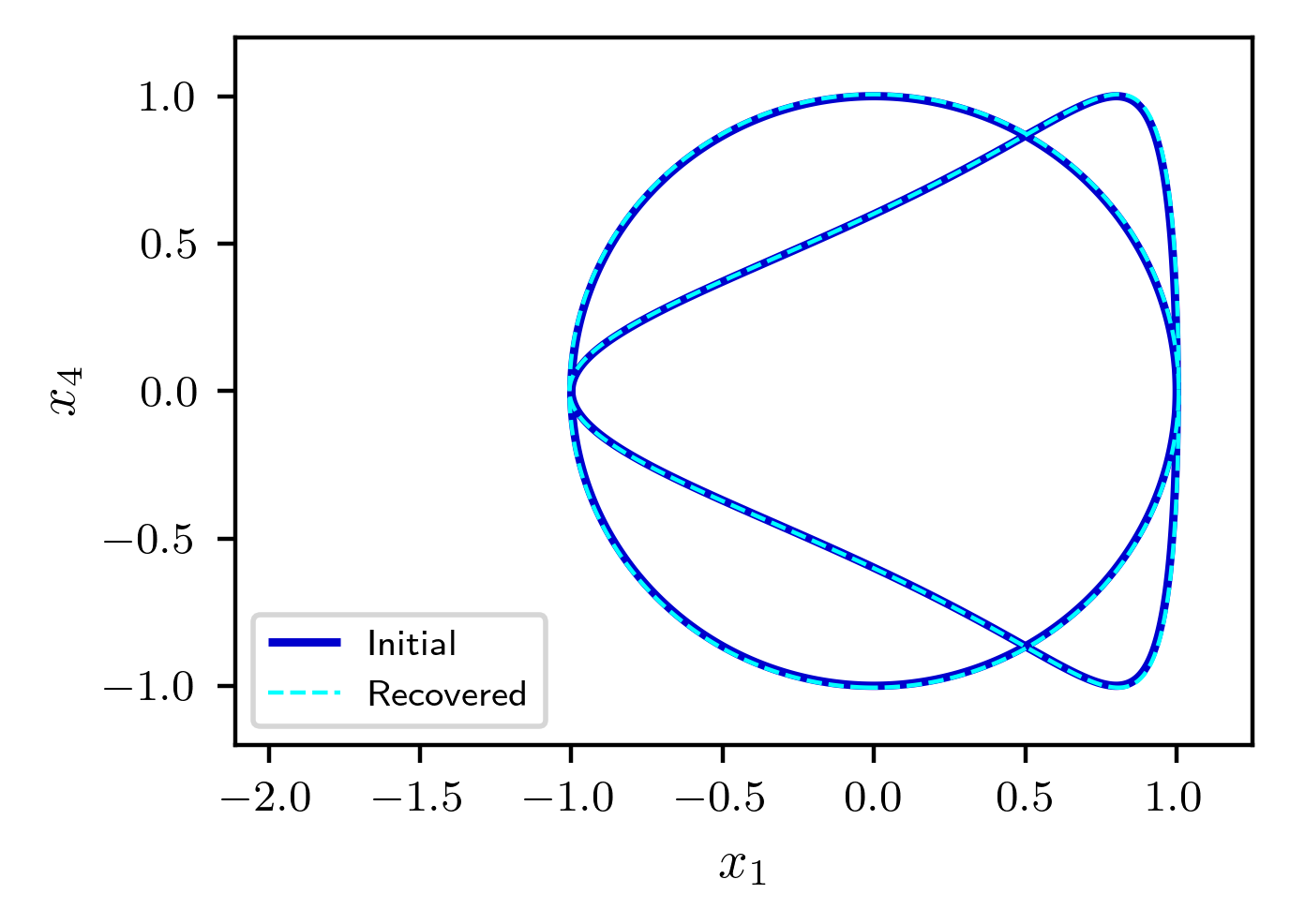} &
    \includegraphics[scale=0.88]{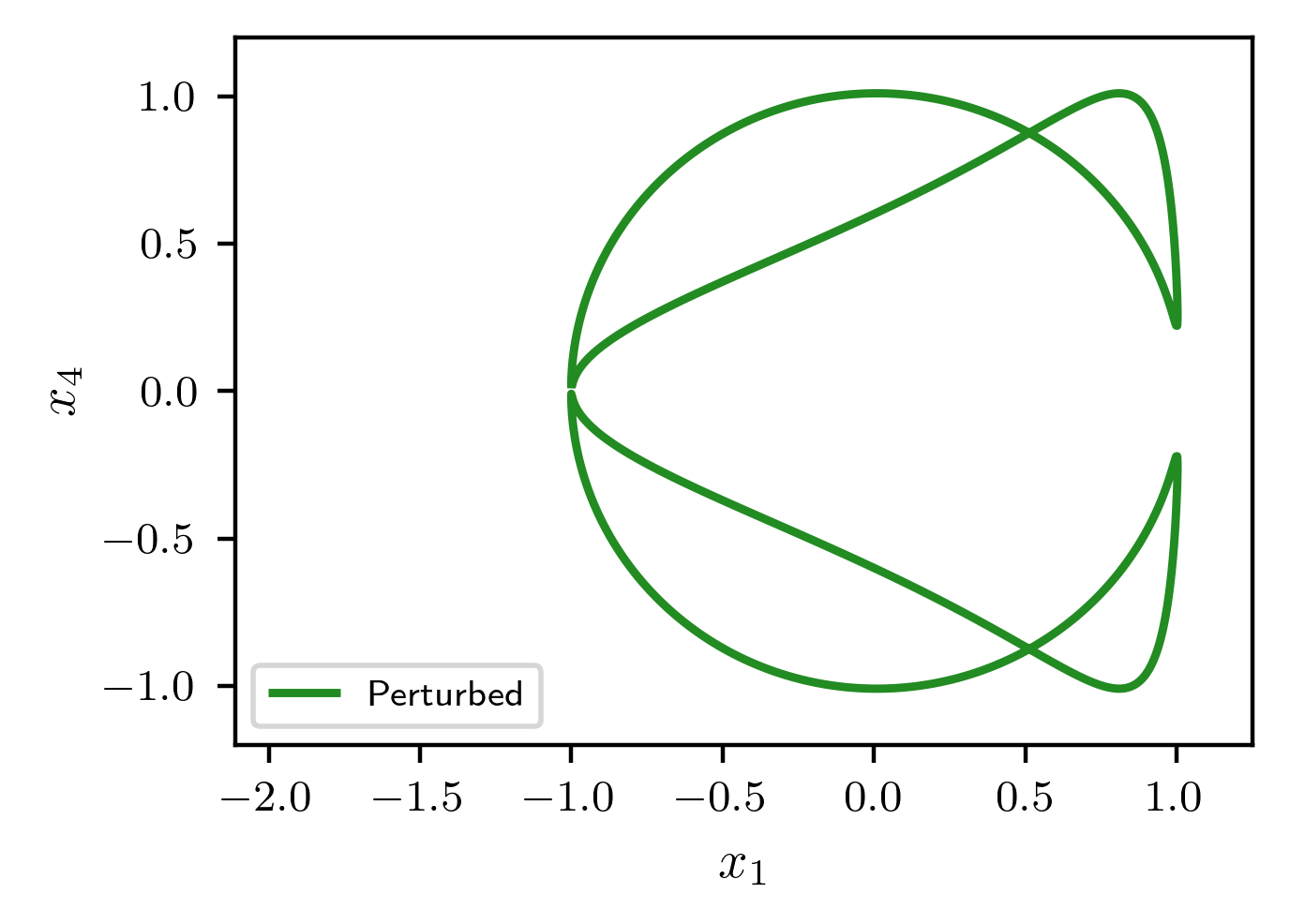} \\
    (a) & (b) \\
    \end{tabular}
    \caption{Projections of the 4-bar coupler curve in $(x_1,x_4)$ space corresponding to (a) initial and recovered parameters, and (b) perturbed parameters}
    \label{fig:4Bar}
\end{figure}

For the perturbed parameter values $\hat{p}$, the
sextic curve does not factor, following Section~\ref{ex:Irreducible}, 
linear traces are close to zero for collections of 
$r=2$ and $d-r=4$ points.  Thus, we aim to apply
Cor.~\ref{co:Irreducible} to recover parameters $p^*$ near $\hat{p}$ such that the solution set~$V(f(x;p^*))$ factors into a quadratic curve and a quartic curve using the second derivative trace test. 
Table~\ref{table:4Bar_LinearAlgebra} summarizes the results
of applying Remark~\ref{remark:Lemma3Dimension} to determine the number of component systems needed, namely the dimension stabilizes
with two systems.  The corresponding system sizes are also reported in 
Table~\ref{table:4Bar_LinearAlgebra}, where the systems are square via homogenized Lagrange multipliers. 
Using a gradient descent homotopy~\eqref{eq:GradDescentHomotopy} with two component systems, the recovered parameters~$p^*$ are reported in Table~\ref{table:4Bar_Parmaeters} to 4 decimals
and one clearly sees the parallelogram linkage structure is recovered.
The resulting
decomposable solution set~$V(f(x;p^*))$ is illustrated in Fig.~\ref{fig:4Bar}(a).

\begin{table}[htb]
\caption{Dimensions and system size for different number of component systems}
\label{table:4Bar_LinearAlgebra}
\centering
    \begin{tabular}[2in]{ccc}
    \toprule
    Component Systems &  Dimension & System Size\\ 
    \midrule
    1 & 3 & \,\,\,53\\
    2 & 2 & 102\\
    3 & 2 & 151\\
    4 & 2 & 200\\
    \bottomrule
    \end{tabular}
\end{table}

\begin{table}[htb]
\caption{Initial (exact), perturbed (4 decimals), and recovered (4 decimals) parameter values}
\label{table:4Bar_Parmaeters}
\centering
    \begin{tabular}[2in]{cccc}
    \toprule
    Parameter & Initial ($\tilde{p}$) & Perturbed ($\hat{p}$) & Recovered ($p^*$) \\
    \midrule
    $p_1$ & 1  & 1.0025  & 1.0062 \\
    $p_2$ & 2  & 2.0101  & 2.0057 \\
    $p_3$ & 1  & 1.0098  & 1.0062 \\
    $p_4$ & 2  & 2.0014  & 2.0057 \\
    \bottomrule
    \end{tabular}
\end{table}

\subsection{Stewart-Gough platform}\label{sec:SGplatform}
A Stewart-Gough platform consists of two bodies, a base and an end-plate, connected by six legs as illustrated in Fig.~\ref{fig:SGplatform}.
For $j=1,\ldots,6$, the $j^{\rm th}$ leg imposes a square
distance $d_j$ between point $a_j\in\Rr^3$ of the base and point $b_j\in\Rr^3$ of the end-plate.
Letting ``$*$'' denote quaternion multiplication and letting $v'$ denote the quaternion conjugate of $v$,
the leg constraints may be written as follows, for $j=1,\ldots,6$,
\begin{multline}\label{eq:SG}
    f_j(e,g;a,b,d) = (a_j * a_j' + b_j * b_j' - d_j)e*e' - e*b_j*e'*a_j'*a_j*e*b_j'*e'\\
    + g*b_j'*e' + e*b_j*g' - g*e'*a_j' - a_j*e*g' + g*g' = 0
\end{multline}
where $e,g$ are quaternions in a Study coordinate representation of the position and orientation of the end-plate. 
Hence, $e,g$ must satisfy the Study quadric
\begin{equation}\label{eq:study}
    Q(e,g) = g_0 e_0 + g_1 e_1 + g_2 e_2 + g_3 e_3 = 0.
\end{equation}
In this example, we set $e_0=1$ to dehomogenize the system. For generic parameters \mbox{$p = (a, b, d)$},
this platform can be assembled in 40 rigid configurations over the complex numbers.  That is, the solution set of the parameterized polynomial system resulting from the 6 leg constraints in \eqref{eq:SG} and Study quadratic in \eqref{eq:study} consists of 40 isolated points. However, for $\tilde{p}$ reported in Table~\ref{table:SG_Parameters} in 
Appendix~\ref{App} derived from \cite[Ex.~2.2]{HSW},
this platform moves in a circular motion
as illustrated in Fig.~\ref{fig:SGplatform}. 
In particular, this circular motion corresponds
to the solution set containing a quadratic curve.
To apply the robustness framework, we consider
a slight perturbation using $\mathcal{N}(0,(10^{-9})^2)$ error yielding $\hat{p}$ in Table~\ref{table:SG_Parameters}
for which the platform becomes rigid.

\begin{figure}[htb]
    \centering
    \includegraphics[scale=0.62]{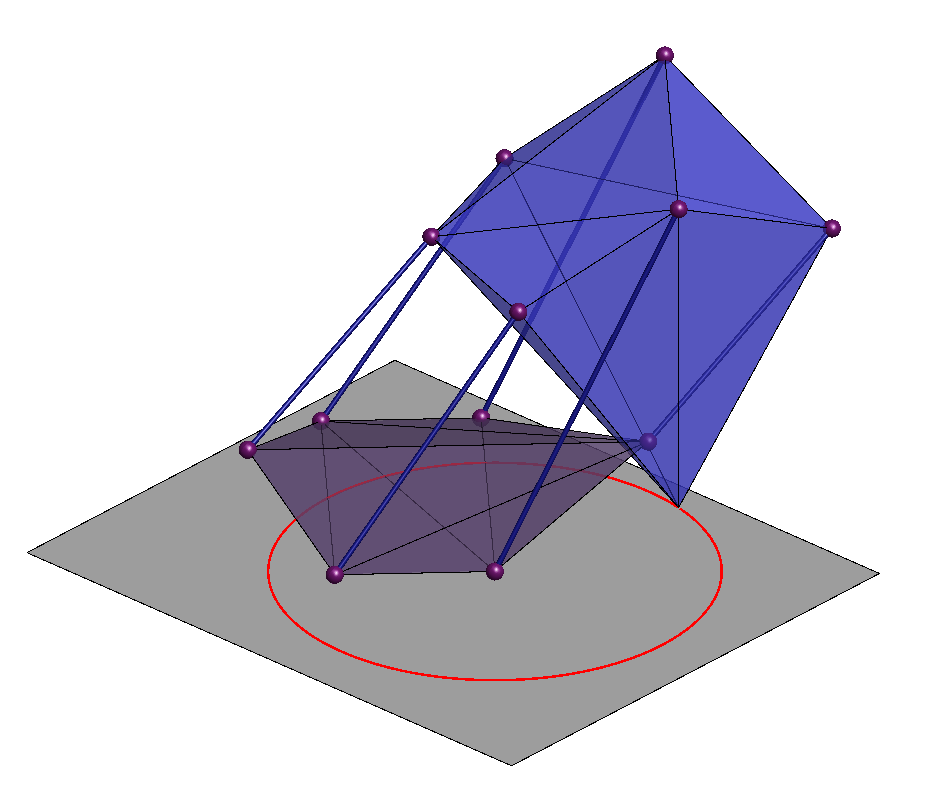} \\
    \caption{A Stewart-Gough platform. The $z=0$ plane (gray) contains the circular path~(red) corresponding to one point of the end-plate}
    \label{fig:SGplatform}
\end{figure}

Following Section \ref{ex:WitnessPoints}, we aim to find $p^*$ near $\hat{p}$ for which the solution set contains a quadratic curve.
First, we consider a randomization of the 6 leg constraints 
down to 5 conditions, the Study quadric, and a linear slice, namely
\begin{equation}
    f_R(e,g;p) = \left[\begin{array}{c}
    f_1 + \Box f_6\\
    f_2 + \Box f_6\\
    f_3 + \Box f_6\\
    f_4 + \Box f_6\\
    f_5 + \Box f_6\\
    Q \\
    \Box e_1 + \Box e_2 + \Box e_3 + \Box g_0 + \Box g_1 + \Box g_2 + \Box g_3 + \Box
    \end{array}\right].
\end{equation}
For $f_R(e,g;\hat{p})$, there are two solutions for which the leg constraints in~\eqref{eq:SG} are close to vanishing, consistent with a degree 2 component having 2 witness points, and
38 solutions which are not close to vanishing.
Using Cor.~\ref{co:WitnessPoints} with $d=2$ to construct the fiber product system, we apply Remark~\ref{remark:Lemma3Dimension} which indicates 4 component systems are needed. The dimensions and corresponding square system sizes via homogenized Lagrange multipliers are reported in Table~\ref{table:SG_LinearAlgebra}. 
Tracking the corresponding gradient descent homotopy~\eqref{eq:GradDescentHomotopy}, 
the recovered parameter values $p^*$ are reported in Table~\ref{table:SG_Parameters} in Appendix~\ref{App}
for which the corresponding Stewart-Gough platform has regained 
its motion.

\begin{table}[htb]
\caption{Dimension and system size for different number of component systems}
\label{table:SG_LinearAlgebra}
\centering
    \begin{tabular}[2in]{ccc}
    \toprule
    Component Systems &  Dimension & System Size\\ 
    \midrule
    1 & 40 & \,\,\,72\\
    2 & 38 & 102\\
    3 & 36 & 132\\
    4 & 35 & 162\\
    5 & 35 & 192\\
    \bottomrule
    \end{tabular}
\end{table}

\subsection{Family containing the 6R inverse kinematics problem}\label{sec:6R}

The inverse kinematics problem for six-revolute (6R) mechanisms seeks to determine all
ways to assemble a loop of six rigid links connected serially by revolute joints. One formulation~\mbox{\cite{morgan1987mHomogeneous,tsai1985solving6R}} sets the problem as a member
of the following parameterized system of eight quadratics using 
a 2-homogeneous construction:
\begin{equation}\label{eq:6R}
    f(x;a) = \left[\begin{array}{c}
        f_0(x;a)\\
        f_1(x;a)\\
        f_2(x;a)\\
        f_3(x;a)\\
        x_1^2 + x_2^2 - x_0^2\\[2pt]
        x_5^2 + x_6^2 - x_0^2\\[2pt]
        x_3^2 + x_4^2 - x_9^2\\[2pt]
        x_7^2 + x_8^2 - x_9^2
    \end{array}\right],
\end{equation}
where $f_j$ has the form
\begin{multline}
    f_j(x;a) = a_{j0} x_1 x_3 + a_{j1} x_1 x_4 + a_{j2} x_2 x_3 + a_{j3} x_2 x_4 + a_{j4} x_5 x_7 + a_{j5} x_5 x_8\\
    {}+ a_{j6} x_6 x_7 + a_{j7} x_6 x_8 + a_{j8} x_1 x_9 + a_{j9} x_2 x_9 + a_{j10} x_3 x_0 + a_{j11} x_4 x_0\\
    {}+ a_{j12} x_5 x_9 + a_{j13} x_6 x_9 + a_{j14} x_7 x_0 + a_{j15} x_8 x_0 + a_{j16} x_0 x_9,
\end{multline}
with $x_0$ and $x_9$ as the homogenizing coordinates. 
In particular, this system is defined on~\mbox{$\Pp^4\times\Pp^4$} with corresponding variable sets
$\{x_0,x_1,x_2,x_5,x_6\}\times\{x_3,x_4,x_7,x_8,x_9\}$. 
We dehomogenize the system by solving on the affine
patches defined by $x_1=1$ and $x_3=1$.

With $a$ consisting of 68 values, we take the parameters
as the $32$ values in $a$ associated with monomials
that do not vanish at infinity, i.e., $V(x_0)\cup V(x_9)$, namely $a_{j0},\dots,a_{j7}$
for $j=0,\dots,3$.
We fix as constants the $36$ values in $a$ associated
with monomials that vanish at infinity,
namely $a_{j8},\dots,a_{j16}$
for $j=0,\dots,3$.  Table~\ref{table:6R_ParametersFixed}
in Appendix~\ref{App} contains the values of these constants
used in the computations.  
For generic parameters, 
the resulting system has 64 finite solutions.  
However, for parameters that correspond with a 6R problem,
the system should only have 32 finite solutions.  
Thus, to utilize the robustness framework, we truncated 
$a\in\Cc^{68}$ corresponding to an actual 6R problem 
using single precision.  Hence, the constants in 
Table~\ref{table:6R_ParametersFixed} are listed in single precision
and the parameter values in Table~\ref{table:6R_Parameters} 
are listed in both single precision (corresponding
to the perturbed parameters) and double precision (corresponding
to the initial parameters).\footnote{
See \url{https://bertini.nd.edu/BertiniExamples/inputIPP_1024}
for values in 1024-bit precision.}
Solving the system with the perturbed parameter values
results in $64$ points corresponding to finite solutions
that are clustered into three groups:
$16$ having $|x_0|$ close to zero,
$16$ having $|x_9|$ close to zero, and
$32$ having both~$|x_0|$ and $|x_9|$ far from zero 
as illustrated in Fig.~\ref{fig:6R_LogPlot}.

\begin{figure}
    \centering
    \includegraphics[scale=0.91]{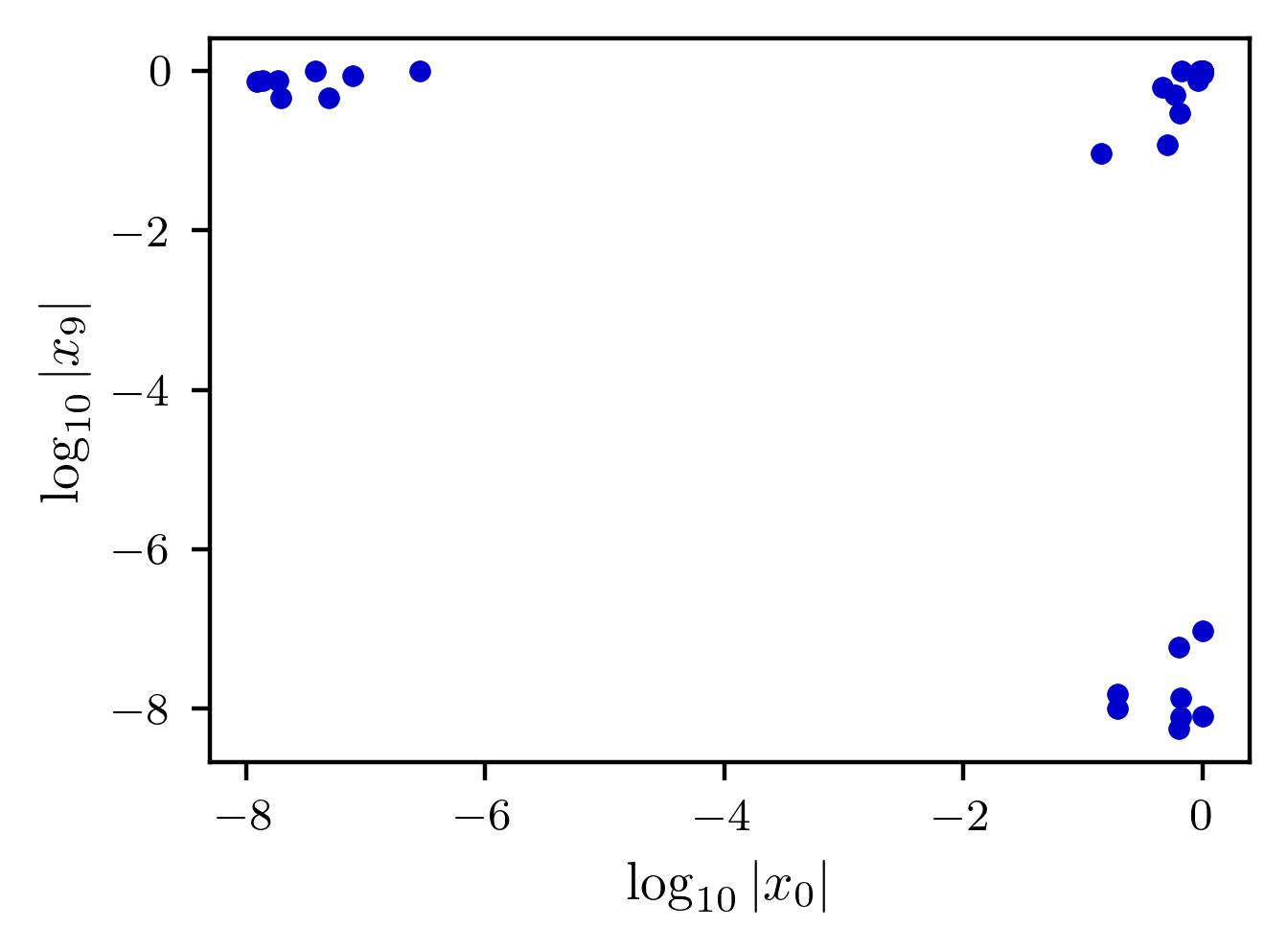}
    \caption{Logarithmic plot of absolute values of homogenizing coordinates for the $64$ solutions}
    \label{fig:6R_LogPlot}
\end{figure}

First, suppose that we aim to recover parameters
by forcing the $16$ solutions with $|x_0|$ close to zero
to be at infinity, i.e., actually satisfy $x_0=0$.
The fiber product system is constructed following
Cor.~\ref{co:SolAtInf}.
Since the solutions are not necessarily independent of each other, 
we utilized Remark~\ref{remark:Lemma3Dimension}
applied to the system with 16 component systems and observed
that there were actually 4 unnecessary conditions,
i.e., only 12 component systems in the fiber product are necessary,
which was confirmed by a gradient descent homotopy~\eqref{eq:GradDescentHomotopy}.

It is the same story if one aims 
to recover parameters
by forcing the $16$ solutions with~$|x_9|$ close to zero
to be at infinity.
So, now suppose that we aim to recover parameters
by forcing both sets of $16$ solutions with either $|x_0|$ 
or $|x_9|$ close to zero to be at infinity.  
Then, applying Remark~\ref{remark:Lemma3Dimension}, we see
that these are not independent and only need $23$
fiber products.  Thus, when pushing these 
$32$ solutions to infinity, we take $12$ for one 
of the infinities and only $11$ for the other.  
After adding homogenized Lagrange multipliers, 
this results in a square system of size 423.
Tracking the gradient descent homotopy~\eqref{eq:GradDescentHomotopy}
yields the recovered parameters reported in Table \ref{table:6R_Parameters} of Appendix~\ref{App}.
Solving with the recovered parameters shows that
all~$32$ of these solutions are pushed back to infinity
as illustrated in Fig.~\ref{fig:6R_ProjectedValues}.

\begin{figure}
    \centering
    \begin{tabular}{cc}
    \includegraphics[height=2.3in]{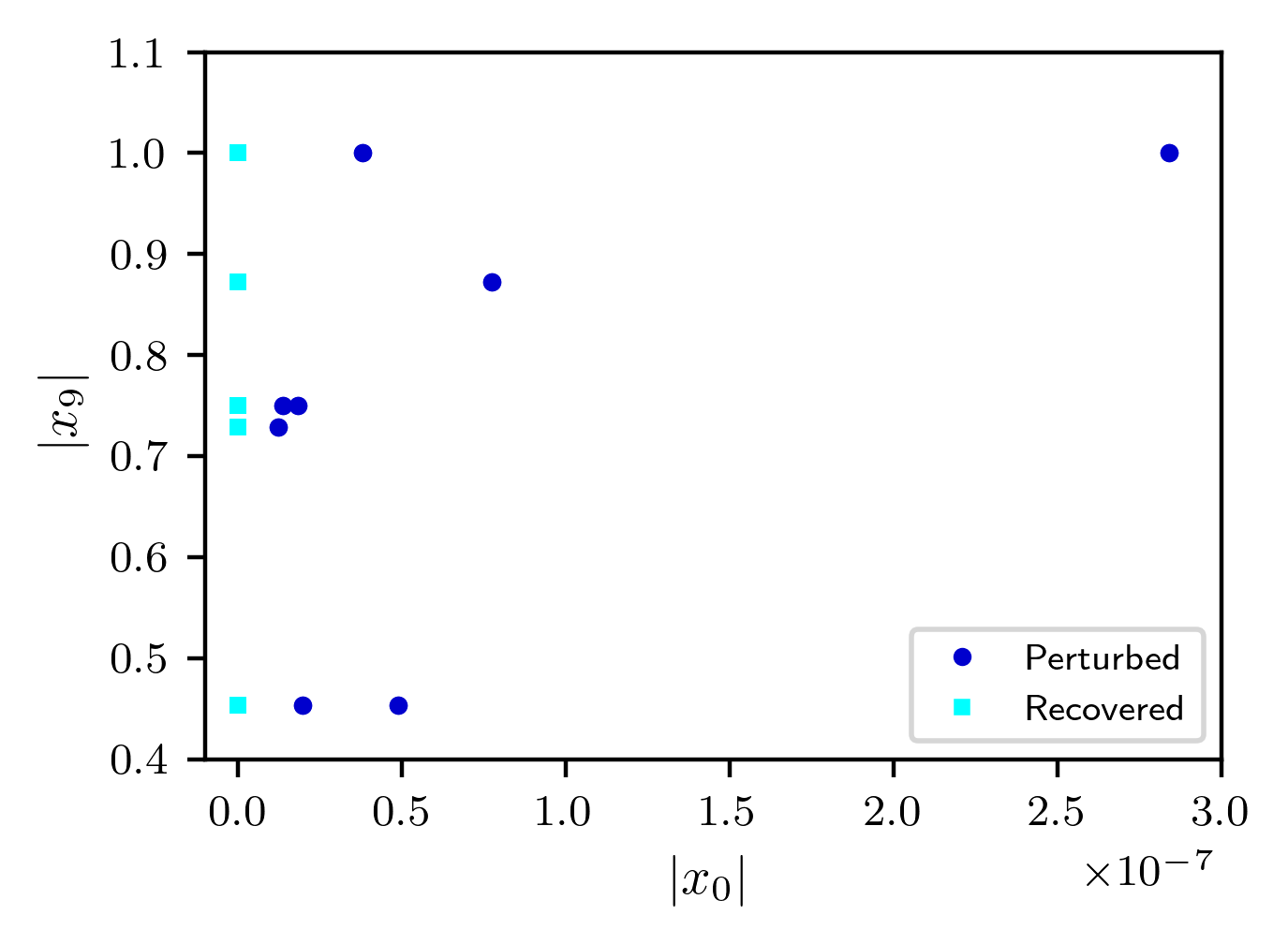} &
    \includegraphics[height=2.3in]{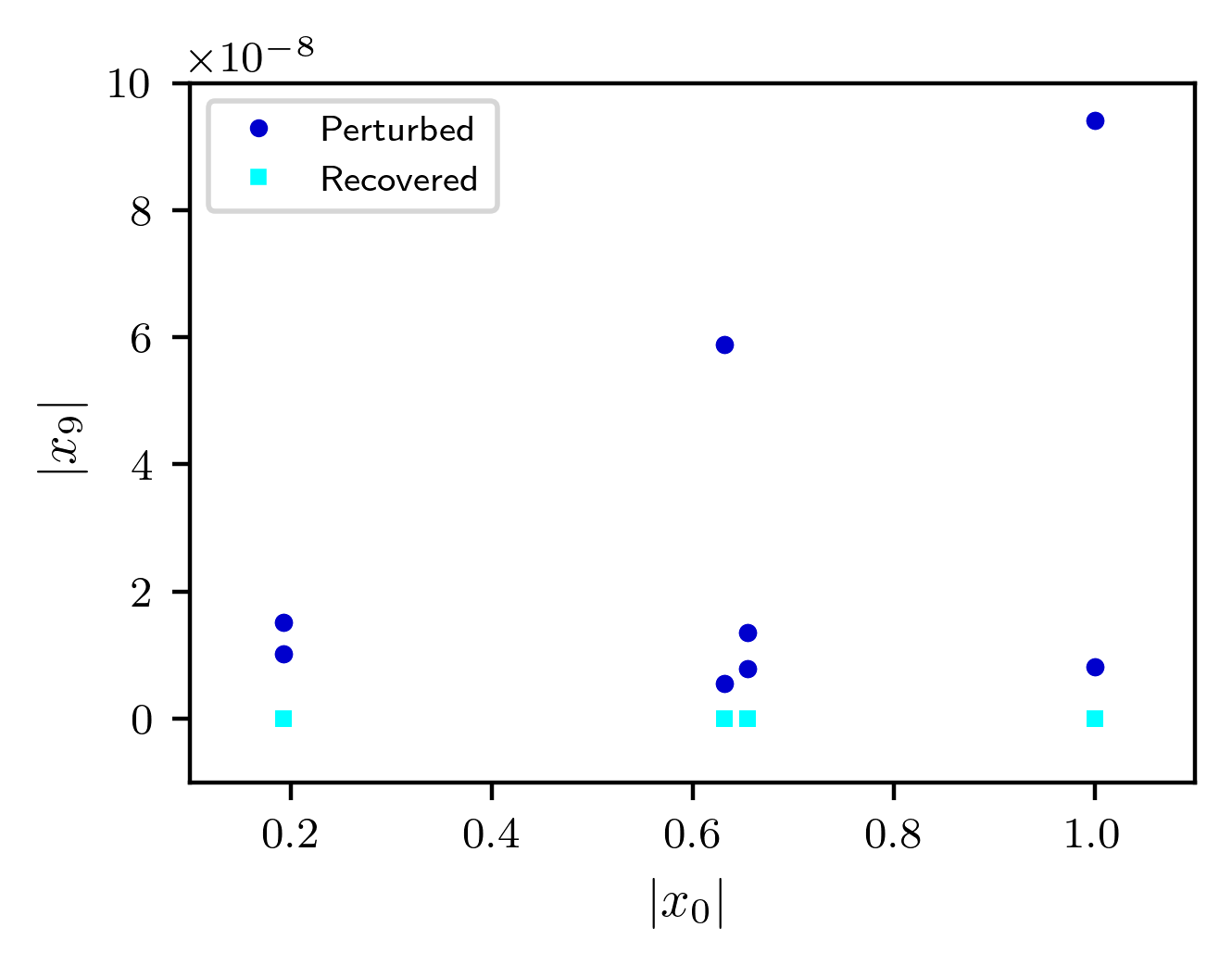} \\
    (a) & (b) \\
    \end{tabular}
    \caption{Absolute values of the homogenizing coordinates for (a) 
    $16$ solutions associated with $x_0=0$ and 
    (b) $16$ solutions associated with $x_9=0$}
    \label{fig:6R_ProjectedValues}
\end{figure}

\section{Conclusion}\label{sec:Conclusion}

After proposing a framework for robustness in numerical algebraic geometry based on fiber products, 
this framework was applied to a collection
of scenarios including fewer finite solutions, 
existence of higher-dimensional components, 
components that further decompose into irreducible components,
and solution sets of higher multiplicity.  
For the cases considered here, with information about the 
questionable structural element already identified,
local optimization techniques applied to fiber products
produced nearby points in parameter space where the
special structure exists.  Moreover, in all of the examples
presented here, we aimed to recover real parameter values
using a gradient descent homotopy based on 
homogenized Lagrange multipliers associated with the Euclidean distance.
For nonreal parameter values, one can utilize isotropic
coordinates.  Also, there are many other optimization
approaches one may use to recover parameter values on exceptional
sets and these other approaches could expand the size of the local convergence zone possibly allowing to recover parameters from larger perturbations.  
Finally, one can adjust the distance metric used
such as one based on knowledge about the relative size of the parameters and their~uncertainties.

Figures~\ref{fig:PostDim_Histograms} and~\ref{fig:histogram_multiplicity}
provide histograms of the recovered parameters arising from
perturbations centered at a parameter value on an exceptional set.
As described in Section~\ref{ex:WitnessPoints}, 
perturbations using a Gaussian distribution will yield 
a Gaussian distribution when projecting onto a linear space.
When projecting onto a nonlinear space, such as in Section~\ref{ex:Multiplicity}, this yields approximately 
a Gaussian distribution on the tangent space.  
Further statistical analysis regarding recovered 
parameters is warranted.

Finally, we note that the codimension of an exceptional set in parameter space may
be less than the number of conditions one seeks to impose on the solution set. 
For example, in the 6R problem of Section~\ref{sec:6R}, one might naively expect
the exceptional set for sending~32 points to infinity to be 
codimension~32,
but in fact, it is codimension 23. While we have found this result numerically,
it raises the more general question of how such results can be understood
using the tools of algebraic geometry.

\section*{Acknowledgment}

The authors were supported in part by
the National Science Foundation CCF-2331440 (ERC and JDH)
and CMMI-2041789 (JDH), 
the Robert and Sara Lumpkins Collegiate Professorship (JDH),
and the Huisking Foundation, Inc. Collegiate Research
Professorship (CWW).

\bibliographystyle{abbrv}
\bibliography{refs,ref_jon}

\appendix
\section{Appendix}\label{App}

The following provides tables associated with
Sections~\ref{sec:SGplatform} and~\ref{sec:6R}.

\begin{table}[htb]
\caption{Initial (exact), perturbed (12 decimals), and recovered (12 decimals) parameters}
\label{table:SG_Parameters}
\centering
\scriptsize
    \begin{tabular}[2in]{cccc}
    \toprule
    Parameter & Initial $(\tilde{p})$ & Perturbed $(\hat{p})$ & Recovered $(p^*)$ \\
    \midrule
    $a_{1_x}$ & 0.0000 & 0.000000000251 & 0.000000000320 \\
    $a_{1_y}$ & 0.0000 & 0.000000001013 & 0.000000000982 \\
    $a_{1_z}$ & 0.0000 & 0.000000000980 & 0.000000000477 \\
    $b_{1_x}$ & 0.0000 & --0.000000000200\,\,\, & --0.000000000269\,\,\, \\ 
    $b_{1_y}$ & 0.0000 & --0.000000000637\,\,\, & --0.000000000606\,\,\, \\
    $b_{1_z}$ & 1.5000 & 1.499999998979 & 1.499999999482 \\
    $d_{1}$   & 3.2500 & 3.249999999891 & 3.249999999724 \\
    
    $a_{2_x}$ & 1.0000 & 1.000000000136 & 1.000000000272 \\
    $a_{2_y}$ & 0.0000 & --0.000000000753\,\,\, & --0.000000001236\,\,\, \\ 
    $a_{2_z}$ & 0.2500 & 0.249999998300 & 0.249999998434 \\
    $b_{2_x}$ & 1.0000 & 1.000000001806 & 1.000000001671 \\
    $b_{2_y}$ & 0.0000 & --0.000000000886\,\,\, & --0.000000000402\,\,\, \\
    $b_{2_z}$ & 1.0000 & 0.999999999658 & 0.999999999525 \\
    $d_{2}$   & 1.5625 & 1.562499999151 & 1.562499999240 \\
    
    $a_{3_x}$ & 1.0000 & 0.999999999712 & 0.999999999839 \\
    $a_{3_y}$ & 1.0000 & 0.999999999733 & 0.999999999918 \\
    $a_{3_z}$ & 0.0000 & 0.000000000109 & --0.000000000010\,\,\, \\ 
    $b_{3_x}$ & 1.0000 & 0.999999998769 & 0.999999998641 \\
    $b_{3_y}$ & 1.0000 & 0.999999999125 & 0.999999998940 \\
    $b_{3_z}$ & 1.5000 & 1.499999999413 & 1.499999999531 \\
    $d_{3}$   & 3.2500 & 3.250000000389 & 3.250000000350 \\

    $a_{4_x}$ & --0.5000\,\,\,  & --0.500000000115\,\,\, & --0.499999999750\,\,\, \\
    $a_{4_y}$ &  0.5000  & 0.500000000098 & 0.500000000171 \\
    $a_{4_z}$ &  0.0000 & 0.000000000167 & 0.000000000893 \\
    $b_{4_x}$ & --0.5000\,\,\,  & --0.499999998762\,\,\, & --0.499999999127\,\,\, \\
    $b_{4_y}$ &  0.5000  & 0.499999999424 & 0.499999999351 \\
    $b_{4_z}$ &  1.0000  & 1.000000000799 & 1.000000000073 \\
    $d_{4}$   &  2.0000  & 2.000000000415 & 2.000000000779 \\

    $a_{5_x}$ &  0.5000  & 0.500000000717 & 0.499999999761 \\
    $a_{5_y}$ &  1.5000  & 1.500000000939 & 1.500000000405 \\
    $a_{5_z}$ &  0.0000  & 0.000000000105 & 0.000000000622 \\
    $b_{5_x}$ &  0.5000  & 0.499999998819 & 0.499999999776 \\
    $b_{5_y}$ &  1.5000  & 1.499999999660 & 1.500000000194 \\
    $b_{5_z}$ &  1.0000  & 0.999999999603 & 0.999999999086 \\
    $d_{5}$   &  2.0000  & 1.999999998435 & 1.999999998693 \\

    $a_{6_x}$ & --0.2500\,\,\, & --0.250000000190\,\,\, & --0.249999999931\,\,\, \\
    $a_{6_y}$ &  1.2500 & 1.249999999364 & 1.250000000155 \\
    $a_{6_z}$ &  0.2500 & 0.250000002270 & 0.250000001515 \\
    $b_{6_x}$ & --0.2500\,\,\, & --0.249999999033\,\,\, & --0.249999999292\,\,\, \\
    $b_{6_y}$ &  1.2500 & 1.250000001020 & 1.250000000228 \\
    $b_{6_z}$ &  1.0000 & 0.999999999682 & 1.000000000437 \\
    $d_{6}$   &  1.5625 & 1.562500000179 & 1.562499999676 \\
    \bottomrule
    \end{tabular}
\end{table}

\begin{table}
\caption{Constant values truncated to single precision}
\label{table:6R_ParametersFixed}
\centering
    \begin{tabular}[2in]{cccc}
    \toprule
    Constant & Single Precision &  Constant & Single Precision \\
    \midrule
    $a_{08}$\,\, & \,\,\,7.4052387$\cdot$10$^{-2}$ & $a_{28}$\,\, & \,\,\,1.9594662$\cdot$10$^{-1}$\\
    $a_{09}$\,\, & --8.3050031$\cdot$10$^{-2}$ & $a_{29}$\,\, & --1.2280341$\cdot$10$^{0}$\,\,\,\,\\
    $a_{010}$ & --3.8615960$\cdot$10$^{-1}$ & $a_{210}$ & \,\,\,0.0000000$\cdot$10$^{0}$\,\,\,\,\\
    $a_{011}$ & --7.5526603$\cdot$10$^{-1}$ & $a_{211}$ & --7.9034219$\cdot$10$^{-2}$\\
    $a_{012}$ & \,\,\,5.0420168$\cdot$10$^{-1}$ & $a_{212}$ & \,\,\,2.6387877$\cdot$10$^{-2}$\\
    $a_{013}$ & --1.0916286$\cdot$10$^{0}$\,\,\,\, & $a_{213}$ & --5.7131429$\cdot$10$^{-2}$\\
    $a_{014}$ & \,\,\,0.0000000$\cdot$10$^{0}$\,\,\,\, & $a_{214}$ & --1.1628081$\cdot$10$^{0}$\,\,\,\,\\
    $a_{015}$ & \,\,\,4.0026384$\cdot$10$^{-1}$ & $a_{215}$ & \,\,\,1.2587767$\cdot$10$^{0}$\,\,\,\,\\
    $a_{016}$ & \,\,\,4.9207289$\cdot$10$^{-2}$ & $a_{216}$ & \,\,\,2.1625749$\cdot$10$^{0}$\,\,\,\,\\
    
    $a_{18}$\,\, & --3.7157270$\cdot$10$^{-2}$ & $a_{38}$\,\, & --2.0816985$\cdot$10$^{-1}$\\
    $a_{19}$\,\, & \,\,\,3.5436895$\cdot$10$^{-2}$ & $a_{39}$\,\, & \,\,\,2.6868319$\cdot$10$^{0}$\,\,\,\,\\
    $a_{110}$ & \,\,\,8.5383480$\cdot$10$^{-2}$ & $a_{310}$ & --6.9910317$\cdot$10$^{-1}$\\
    $a_{111}$ & \,\,\,0.0000000$\cdot$10$^{0}$\,\,\,\, & $a_{311}$ & \,\,\,3.5744412$\cdot$10$^{-1}$\\
    $a_{112}$ & --3.9251967$\cdot$10$^{-2}$ & $a_{312}$ & \,\,\,1.2499117$\cdot$10$^{0}$\,\,\,\,\\
    $a_{113}$ & \,\,\,0.0000000$\cdot$10$^{0}$\,\,\,\, & $a_{313}$ & \,\,\,1.4677360$\cdot$10$^{0}$\,\,\,\,\\
    $a_{114}$ & --4.3241927$\cdot$10$^{-1}$ & $a_{314}$ & \,\,\,1.1651719$\cdot$10$^{0}$\,\,\,\,\\
    $a_{115}$ & \,\,\,0.0000000$\cdot$10$^{0}$\,\,\,\, & $a_{315}$ & \,\,\,1.0763397$\cdot$10$^{0}$\,\,\,\,\\
    $a_{116}$ & \,\,\,1.3873009$\cdot$10$^{-2}$ & $a_{316}$ & --6.9686807$\cdot$10$^{-1}$\\ 
    \bottomrule
    \end{tabular}
\end{table}

\begin{table}
\caption{Initial (double precision), perturbed (truncated single precision), and recovered (double precision) parameters}
\label{table:6R_Parameters}
\centering
    \begin{tabular}[2in]{ccc}
    \toprule
    Parameter & Single$\vert$Double Precision & Recovered \\
    \midrule
    $a_{00}$ & --2.4915068$\vert$11232596$\cdot$10$^{-1}$ & --2.491506848757833$\cdot$10$^{-1}$\\
    $a_{01}$ & \,\,\,1.6091353$\vert$78745045$\cdot$10$^{0}$\,\,\: & \,\,\,1.609135324728055$\cdot$10$^{0}$\,\,\:\\
    $a_{02}$ & \,\,\,2.7942342$\vert$61384628$\cdot$10$^{-1}$ & \,\,\,2.794234123846178$\cdot$10$^{-1}$\\
    $a_{03}$ & \,\,\,1.4348015$\vert$88307759$\cdot$10$^{0}$\,\,\: & \,\,\,1.434801543598025$\cdot$10$^{0}$\,\,\:\\ 
    $a_{04}$ & \,\,\,0.0000000$\vert$00000000$\cdot$10$^{0}$\,\,\: & \,\,\,2.329107073061927$\cdot$10$^{-8}$\\
    $a_{05}$ & \,\,\,4.0026384$\vert$20852447$\cdot$10$^{-1}$ & \,\,\,4.002638399151275$\cdot$10$^{-1}$\\
    $a_{06}$ & --8.0052768$\vert$41704895$\cdot$10$^{-1}$ & --8.005276506597172$\cdot$10$^{-1}$\\
    $a_{07}$ & \,\,\,0.0000000$\vert$00000000$\cdot$10$^{0}$\,\,\: & \,\,\,1.339330350134300$\cdot$10$^{-8}$\\ 
    
    $a_{10}$ & \,\,\,1.2501635$\vert$03697273$\cdot$10$^{-1}$ & \,\,\,1.250163518996785$\cdot$10$^{-1}$\\
    $a_{11}$ & --6.8660735$\vert$90276054$\cdot$10$^{-1}$ & --6.866073304900900$\cdot$10$^{-1}$\\
    $a_{12}$ & --1.1922811$\vert$66678474$\cdot$10$^{-1}$ & --1.192281095708419$\cdot$10$^{-1}$\\ 
    $a_{13}$ & --7.1994046$\vert$84195284$\cdot$10$^{-1}$ & --7.199404481832083$\cdot$10$^{-1}$\\ 
    $a_{14}$ & --4.3241927$\vert$30334479$\cdot$10$^{-1}$ & --4.324192773933984$\cdot$10$^{-1}$\\ 
    $a_{15}$ & \,\,\,0.0000000$\vert$00000000$\cdot$10$^{0}$\,\,\: & \,\,\,1.358542627603532$\cdot$10$^{-8}$\\
    $a_{16}$ & \,\,\,0.0000000$\vert$00000000$\cdot$10$^{0}$\,\,\: & --1.039184803095887$\cdot$10$^{-9}$\\
    $a_{17}$ & --8.6483854$\vert$60668959$\cdot$10$^{-1}$ & --8.648385383114613$\cdot$10$^{-1}$\\

    $a_{20}$ & --6.3555007$\vert$06536143$\cdot$10$^{-1}$ &  --6.355500280163283$\cdot$10$^{-1}$\\
    $a_{21}$ & --1.1571992$\vert$24063992$\cdot$10$^{-1}$ &  --1.157199361445811$\cdot$10$^{-1}$\\
    $a_{22}$ & --6.6640447$\vert$34656436$\cdot$10$^{-1}$ &  --6.664044436579097$\cdot$10$^{-1}$\\
    $a_{23}$ & \,\,\,1.1036211$\vert$15850889$\cdot$10$^{-1}$ &  \,\,\,1.103620867759053$\cdot$10$^{-1}$\\
    $a_{24}$ & \,\,\,2.9070203$\vert$22913935$\cdot$10$^{-1}$ &  \,\,\,2.907020211729024$\cdot$10$^{-1}$\\
    $a_{25}$ & \,\,\,1.2587767$\vert$24480555$\cdot$10$^{0}$\,\,\: &  \,\,\,1.258776710166779$\cdot$10$^{0}$\,\,\:\\
    $a_{26}$ & --6.2938836$\vert$22402776$\cdot$10$^{-1}$ &  --6.293883708977084$\cdot$10$^{-1}$\\
    $a_{27}$ & \,\,\,5.8140406$\vert$45827871$\cdot$10$^{-1}$ &  \,\,\,5.814040462810132$\cdot$10$^{-1}$\\

    $a_{30}$ & \,\,\,1.4894773$\vert$41316300$\cdot$10$^{0}$\,\,\: &  \,\,\,1.489477303748473$\cdot$10$^{0}$\,\,\:\\
    $a_{31}$ & \,\,\,2.3062341$\vert$36720304$\cdot$10$^{-1}$ &  \,\,\,2.306233954795566$\cdot$10$^{-1}$\\
    $a_{32}$ & \,\,\,1.3281073$\vert$07376312$\cdot$10$^{0}$\,\,\: &  \,\,\,1.328107268535429$\cdot$10$^{0}$\,\,\:\\
    $a_{33}$ & --2.5864502$\vert$59957599$\cdot$10$^{-1}$ &  --2.586450384436285$\cdot$10$^{-1}$\\
    $a_{34}$ & \,\,\,1.1651719$\vert$51133394$\cdot$10$^{0}$\,\,\: &  \,\,\,1.165171916593329$\cdot$10$^{0}$\,\,\:\\
    $a_{35}$ & --2.6908493$\vert$58556267$\cdot$10$^{-1}$ &  --2.690849292497942$\cdot$10$^{-1}$\\
    $a_{36}$ & \,\,\,5.3816987$\vert$17112534$\cdot$10$^{-1}$ &  \,\,\,5.381698714725988$\cdot$10$^{-1}$\\
    $a_{37}$ & \,\,\,5.8258597$\vert$55666972$\cdot$10$^{-1}$ &  \,\,\,5.825859575485448$\cdot$10$^{-1}$\\    
    \bottomrule
    \end{tabular}
\end{table}

\end{document}